\documentclass[12pt,reqno]{amsproc}
\usepackage{}
\usepackage{mathrsfs, amsfonts, amsthm, amsmath, amssymb}

\theoremstyle{plain}

\newtheorem{theorem}{Theorem}[subsection]

\newtheorem{corollary}[theorem]{Corollary}

\newtheorem{definition}[theorem]{Definition}
\newtheorem{example}[theorem]{Example}

\newtheorem{lemma}[theorem]{Lemma}

\newtheorem{proposition}[theorem]{Proposition}
\newtheorem{remark}[theorem]{Remark}

\makeatletter
\newcommand{\LeftEqNo}{\let\veqno\@@leqno}

\makeatother

 \numberwithin{equation}  {section}
\begin{document}

\

\vspace{-2cm}

\title[Kato-Rosenblum theorem in semifinite von Neumann algebras]{Perturbations of self-adjoint operators in semifinite von Neumann algebras: Kato-Rosenblum theorem}
\author{Qihui Li}
\curraddr{School of Science, East China University of Science and Technology, Shanghai,
200237, P. R. China}
\email{qihui\_{}li@126.com}
\author{Junhao Shen}
\curraddr{Department of Mathematics \& Statistics, University of New Hampshire, Durham,
03824, US}
\email{Junhao.Shen@unh.edu}
\author{Rui Shi}
\curraddr{School of Mathematical Sciences, Dalian University of Technology, Dalian,
116024, P. R. China}
\email{ruishi@dlut.edu.cn, ruishi.math@gmail.com}
\author{Liguang Wang}
\curraddr{School of Mathematical Sciences, Qufu Normal University, Qufu, 273165, P. R. China}
\email{wangliguang0510@163.com}
\thanks{The first author was partly supported by NSFC(Grant No.11671133).}
\thanks{}
\thanks{The third author was partly supported by NSFC(Grant
No.11401071) and the Fundamental Research Funds for the Central Universities
(Grant No.DUT16RC(4)57).}
\thanks{The fourth author was partly supported by NSFC(Grant No.11371222 and No.11671133).}
\subjclass[2010]{Primary: 47C15; Secondary: 47L60,47L20}
\keywords{The generalized wave operators, The Kato-Rosenblum theorem, Norm-ideal perturbations, von Neumann algebras}
\begin{abstract}
In the paper, we prove an analogue of the Kato-Rosenblum theorem in  a semifinite von Neumann algebra. Let $\mathcal{M}$ be a countably decomposable, properly infinite, semifinite von Neumann algebra acting  on a Hilbert space $\mathcal{H}$ and let $\tau$ be a faithful normal semifinite tracial weight of $\mathcal M$. Suppose that $H$ and $H_1$ are self-adjoint operators affiliated with $\mathcal{M}$. We show that
if $H-H_1$ is in $\mathcal{M}\cap L^{1}\left(
\mathcal{M},\tau\right)$, then the {\em norm} absolutely continuous parts of $H$ and $H_1$  are unitarily equivalent. This implies that the real part of a non-normal hyponormal operator in $\mathcal M$ is not a perturbation  by $\mathcal{M}\cap L^{1}\left(
\mathcal{M},\tau\right)$ of a diagonal operator. Meanwhile, for $n\ge 2$ and  $1\leq p<n$, by modifying Voiculescu's invariant we give   examples of   commuting
$n$-tuples of self-adjoint operators in $\mathcal{M}$ that are not
arbitrarily small perturbations of commuting diagonal operators modulo
$\mathcal{M}\cap L^{p}\left(  \mathcal{M},\tau\right)  $.

\end{abstract}
\maketitle




\section{Introduction}

 This   paper is a continuation of the investigation, which we began in
 \cite{Li}, of diagonalizations of self-adjoint operators  modulo
 norm ideals in   semifinite von Neumann algebras.

Let $\mathcal{H}_0$ be a complex separable infinite dimensional Hilbert space. Denote by
$\mathcal{B}\left(  \mathcal{H}_0\right)  $ the set of bounded linear operators
on $\mathcal{H}_0$.
The Weyl-von Neummann theorem  \cite{Weyl,Von2} states that a self-adjoint operator  in $\mathcal{B}\left(  \mathcal{H}_0\right)  $  is a
sum of a diagonal operator and an arbitrarily small Hilbert-Schmidt operator. A result by Kuroda in  \cite{Kuroda} implies that  a self-adjoint operator in $\mathcal{B}\left(  \mathcal{H}_0\right)  $ is a
sum of a diagonal operator and an arbitrarily small Schatten $p$-class operator with $p>1$.  Berg and Sikonia independently showed in \cite{Berg} and \cite{Siknia} that a normal operator in $\mathcal{B}\left(  \mathcal{H}_0\right)  $ is a compact perturbation of a diagonal operator. In \cite{Voi}, Voiculescu proved a surprising result by showing that a normal operator in $\mathcal{B}\left(  \mathcal{H}_0\right)  $ is a diagonal operator plus an arbitrarily small Hilbert-Schmidt operator.
 This  result of Voiculescu has recently been generalized in  \cite{Li} to    semifinite von Neumann algebras with  separable predual. It is worth noting that   Kuroda's result in \cite{Kuroda} was also extended  to   countably decomposable, properly infinite, semifinite von Neumann algebras in \cite{Li}.

In the case of perturbations by trace class operators,  the influential Kato-Rosenblum theorem (see \cite{Kato} and \cite{Rosenblum})  provides an obstruction to diagonalizations, modulo the trace class, of self-adjoint operators in $\mathcal B(\mathcal H_0)$. More specifically, if $H$ and $H_1$ are densely defined self-adjoint operators on $\mathcal H_0$ such that $H-H_1$ is in the trace class, then the Kato-Rosenblum theorem asserts that the absolutely continuous parts of $H$ and $H_1$ are unitarily equivalent.  
Thus, if a self-adjoint operator $H$ in $\mathcal B(\mathcal H_0)$
has a nonzero absolutely continuous spectrum, then $H$ can not be a
sum of a diagonal operator and a trace class operator.

The purpose of this paper is to provide a version of the
Kato-Rosenblum theorem in a   semifinite von Neumann algebra.   (For
general knowledge about von Neumann algebras, the reader is referred
to \cite{Dixmier,Kadison}.) Let $\mathcal H$ be a complex infinite
dimensional Hilbert space and let $\mathcal M\subseteq \mathcal
B(\mathcal H)$ be a countably decomposable, properly infinite von
Neumann algebra  with a faithful normal tracial weight $\tau$. A
quick example (see Example \ref{example2.4.2}) shows the existence
of a self-adjoint operator $A$ in $\mathcal M$ satisfying that $A$
has a nonzero absolutely continuous spectrum and $A$ is also  a sum
of a diagonal operator and an arbitrarily small  operator in
$\mathcal M\cap L^1(\mathcal M,\tau)$. Thus we should not expect
that  a direct generalization of  the Kato-Rosenblum theorem still holds
in a general semifinite von Neumann algebra $\mathcal M$.

Before   stating the results of the paper, we recall   the following notation. Let $(\mathscr X,\Vert\cdot\Vert)$ be a Banach space. A mapping
$f:\mathbb{R}\rightarrow\mathscr X$ is \emph{locally absolutely continuous}
if, for all $a,b\in\mathbb{R}$ with $a<b$ and every $\epsilon>0$, there exists
a $\delta>0$ such that $\sum_{i}\Vert f({b_{i}})-f({a_{i}})\Vert<\epsilon$ for
every finite collection $\{(a_{i},b_{i})\}$ of disjoint intervals in $[a,b]$
with $\sum_{i}(b_{i}-a_{i})<\delta.$

In this paper, we introduce a notion of {\em norm absolutely continuous projections} with respect to a self-adjoint operator $H$  affiliated with $\mathcal M$. Suppose $\{E(\lambda)\}_{\lambda\in\mathbb R}$ is the spectral resolution of the identity for $H$ in $\mathcal M$. A projection $P$ in $\mathcal M$ is called a norm absolutely continuous projection   with respect to  $H$ if the mapping $\lambda \mapsto PE(\lambda)P$ from $\mathbb R$ into $\mathcal M$ is locally absolutely continuous (see Definition \ref{def5.1.1}). It is shown in Proposition \ref{example5.1.3} that, in the case of $\mathcal M=\mathcal B(\mathcal H)$, if $x\in\mathcal H$ and $x\otimes x$ is the rank one projection associated with $x$, then $x\otimes x$ is a norm absolutely continuous projection  with respect to  $H$ if and only if the vector $x$ is  absolutely continuous with respect to $H$.

 For a self-adjoint operator $H$  affiliated with $\mathcal M$, we define the {\em norm absolutely continuous support} $P_{ac}^\infty(H)$ of $H$ to be the union of these norm absolutely continuous projections  with respect to  $H$ (see Definition \ref{def5.2.1}). When $H\in\mathcal M$ is
bounded, the following  criterion gives a characterization of
$P_{ac}^\infty(H)$ in terms of    hyponormal operators in $\mathcal
M$.

 \vspace{0.2cm}

{C{\Small OROLLARY}} \ref{prop5.2.10}
{\em $H$ is a self-adjoint element in $\mathcal
M$ with  $P_{ac}^\infty(H)\ne 0$ if and only if $H$ is the real part
of a non-normal hyponormal operator $T$ in $\mathcal M$.}

 \vspace{0.2cm}

 Now we are ready to state our analogue of the Kato-Rosenblum theorem for a semifinite von Neumann algebra.

 \vspace{0.2cm}

{T{\Small HEOREM}} \ref{thm5.2.2} {\em Suppose $H$ and $H_{1}$ are  self-adjoint operators affiliated with
$\mathcal{M}$ such that $H_{1} - H$ is in
$\mathcal{M}\cap L^{1}(\mathcal{M},\tau)$. Then
\[
W_{+}\triangleq sot\text{-}\lim_{t\rightarrow\infty}e^{ itH_{1}}e^{-itH}
P_{ac}^{\infty}(H) \ \ \text{ exists in } \ \mathcal{M}.
\]
Moreover,
\begin{enumerate}
\item [(i)]$\displaystyle
W_{+}^{*}W_{+}= P_{ac}^{\infty}(H) $ and $ \displaystyle W_{+} W_{+}^{*}= P_{ac}^{\infty}(H_{1});
$
\item [(ii)]
$\displaystyle
W_+   H  W_+^*=   H_1 P_{ac}^{\infty}(H_{1})  .
$
\end{enumerate}
}

 \vspace{0.2cm}

 A direct consequence of Theorem \ref{thm5.2.2} is the next result.

 \vspace{0.2cm}

 {P{\Small ROPOSITION}} \ref{prop5.3.4}
\label{cor5.2.5} {\em If $H$ is a self-adjoint element in $\mathcal{M}$ such that $P_{ac}%
^{\infty}(H)\ne0$, then there exists no self-adjoint diagonal
operator $K$ in $\mathcal{M}$ satisfying $H-K\in
L^{1}(\mathcal{M},\tau)$. In particular, if $H$ is  the real part of a
non-normal hyponormal operator   in $\mathcal M$, then there exists no self-adjoint diagonal
operator $K$ in $\mathcal{M}$ satisfying $H-K\in
L^{1}(\mathcal{M},\tau)$.
 }

 \vspace{0.2cm}

We are also able to obtain an analogue of the Kuroda-Birman theorem for a semifinite von Neumann algebra as follows.

 \vspace{0.2cm}

{T{\Small HEOREM}} \ref{thm5.3.2} {\em Suppose $H$ and $H_{1}$ are  self-adjoint operators  affiliated with
$\mathcal{M}$ such that
\[
(H_{1}+i)^{-1} - (H+i)^{-1} \in\mathcal{M}\cap L^{1}(\mathcal{M},\tau).
\]
Then
\[
W_{+}\triangleq sot\text{-}\lim_{t\rightarrow\infty}e^{ itH_{1}}e^{-itH}
P_{ac}^{\infty}(H) \ \ \text{ exists in } \ \mathcal{M}.
\]
Moreover,
\begin{enumerate}
\item [(i)]$\displaystyle
W_{+}^{*}W_{+}= P_{ac}^{\infty}(H) $ and $ \displaystyle W_{+} W_{+}^{*}= P_{ac}^{\infty}(H_{1});
$
\item [(ii)]
$\displaystyle
W_+   H  W_+^*=   H_1 P_{ac}^{\infty}(H_{1})  .
$
\end{enumerate}
}

For a commuting $n$-tuple  of self-adjoint operators in
$\mathcal{B}\left(  \mathcal{H}_0\right)  $ as $n\geq2$, the
simultaneous diagonalization theory has been extensively
investigated  in \cite{BV, Berg}, \cite{Voi2}-\cite{Voigt},
\cite{Xia1}-\cite{Xia5}. In this paper, we consider obstructions
to simultaneous diagonalization of self-adjoint operators in a
countably decomposable, proper infinite von Neumann algebra
$\mathcal M$ with a faithful normal tracial weight $\tau$. By
modifying Voiculescu's invariant, in Example \ref{example6.2.2} we
exhibit  an example   of an $n$-tuple of commuting self-adjoint
operators in   $\mathcal{M}$ that is not an arbitrarily small
perturbation of commuting diagonal operators modulo $\mathcal{M}\cap
L^{p}\left( \mathcal{M},\tau\right)  $ for all $1\leq p<n.$

The present paper has six sections. In section 2, we prepare related notation,
definitions and lemmas. We recall the concept of absolutely continuous
spectrum and give an example of a purely absolutely continuous self-adjoint
operator in a semifinite von Neumann algebras that is an arbitrarily small
$\max\{\Vert\cdot\Vert,\Vert\cdot\Vert_{1}\}$-norm perturbation of a diagonal
operator. In section 3, we introduce a smooth condition for a densely defined self-adjoint operator. Under this condition,
we are able to give the point-wise convergence of generalized wave operators. Norm
absolutely continuous projections with respect to a self-adjoint operator $H$
affiliated with  $\mathcal{M}$ are introduced in section 4. Section 5 is devoted
to show an analogue of the Kato-Rosenblum Theorem in semifinite von Neumann algebras. We also
provide  an analogue of  the Kuroda-Birman Theorem in a semifinite von Neumann algebra.
Section 6 provides examples of $n$-tuple of
self-adjoint operators in   $\mathcal{M}$ that can not be an arbitrary
small perturbations of commuting diagonal operators modulo $\mathcal{M}\cap
L^{p}\left(  \mathcal{M},\tau\right)  $ for all $1\leq p<n.$

\section{Preliminaries and Notation}

Let $\mathcal{H}$ be a complex Hilbert space and let $\mathcal{B}%
(\mathcal{H})$ be the set of all bounded linear operators on
$\mathcal{H}$.

\subsection{Semifinite von Neumann algebra}

Let $\mathcal{M}\subseteq\mathcal{B}(\mathcal{H})$ be a countably
decomposable, properly infinite von Neumann algebra with a faithful normal
semifinite tracial weight $\tau$ (see Definition 7.5.1 in \cite{Kadison} for
the details). Let
\[
\begin{aligned} \mathscr F(\mathcal M,\tau) &= \{AEB \ : \ E=E^*=E^2\in \mathcal M \text { with $\tau(E)<\infty$ and } A,B\in\mathcal M\}\\ \end{aligned}
\]
be the set of finite rank operators in $(\mathcal{M},\tau)$.

The following result is well-known. For the purpose of completeness, we
include its proof here.

\begin{lemma}
\label{lemma3.1.1} Let $\mathcal{M}\subseteq\mathcal{B}(\mathcal{H})$ be a
countably decomposable, properly infinite von Neumann algebra with a faithful
normal semifinite tracial weight $\tau$.

\begin{enumerate}
\item[(i)] There exists a sequence $\{P_{n}\}_{n\in\mathbb{N}}$ of orthogonal
projections in $\mathcal{M}$ such that $\tau(P_{n})<\infty$ for each
$n\in\mathbb{N}$ and $\sum_{n\in\mathbb N} P_{n}=I$ (convergence is
in strong operator topology).

\item[(ii)] There exists a sequence $\{x_{m}\}_{m\in\mathbb{N}}$ of vectors in
$\mathcal{H}$ such that
\[
\tau(X^*X)=\sum_{m}\langle X^*Xx_{m}, x_{m}\rangle, \qquad\forall\
X\in \mathcal{M}.
\]
  Moreover, the
linear span of the set $\{ A^{\prime}x_{m} : A^{\prime}\in\mathcal{M}^{\prime
}\text{ and } m\in\mathbb{N}\}$ is dense in $\mathcal{H}$, where
$\mathcal{M}^{\prime}$ is the commutant of $\mathcal{M}$ in $\mathcal{B}%
(\mathcal{H})$.
\end{enumerate}
\end{lemma}

\begin{proof}
From Proposition 8.5.2 in \cite{Kadison}, if $P$ is a nonzero projection in
$\mathcal{M}$, then there exists a sub-projection $P_{0}$ of $P$ such that
$0<\tau( P_{0})<\infty$. Now by Zorn's lemma, there exists a family
$\{P_{\lambda}\}_{\lambda\in\Lambda}$ of orthogonal projections in
$\mathcal{M}$ such that $\sum_{\lambda\in\Lambda} P_{\lambda}=I$ and
$\tau(P_{\lambda})<\infty$ for each $\lambda\in\Lambda$. From the fact that
$\mathcal{M}$ is countably decomposable, it follows that $\Lambda$ is
countable. This ends the proof of (i).

By (i), there exists a sequence $\{P_{n}\}_{n\in\mathbb{N}}$ of
orthogonal projections in $\mathcal{M}$ such that
$\tau(P_{n})<\infty$ for each $n\in\mathbb{N}$ and $\sum_{n}
P_{n}=I$. Thus $\tau(A)=\sum_{n} \tau(AP_{n}),$ $\forall\
A\in\mathcal{M}^{+}$, where $\mathcal M^+$ is the positive part of
$\mathcal M$. As the mapping $A\mapsto\tau(AP_{n})$ is a
normal positive functional on $\mathcal{M}$, the existence of $\{x_{m}%
\}_{m\in\mathbb{N}}$ follows from Theorem 7.1.12 in \cite{Kadison}.
Moreover, let $Q$ be the projection from $\mathcal{H}$ onto
$\mathcal{H}_{1}$, the
closure of the linear span of the set $\{ A^{\prime}x_{m} : A^{\prime}%
\in\mathcal{M}^{\prime}\text{ and } m\in\mathbb{N}\}$ in
$\mathcal{H}$. Then $Q\in\mathcal{M}$ and $(I-Q)x_{m}=0$ for all
$m\in\mathbb{N}$. Thus $\tau(I-Q)=\sum_{m}\langle(I-Q)x_{m},
x_{m}\rangle=0$. As $\tau$ is faithful, we conclude that $I-Q=0$,
whence $\mathcal{H}_{1}=\mathcal{H}$.
\end{proof}

\subsection{Noncommutative $L^{p}(\mathcal{M},\tau)$}

Here, we will briefly review the definition of noncommutative $L^{p}$-spaces
associated to a semifinite von Neumann algebra. For $1\le p<\infty$, the
mapping
\[
\|\cdot\|_{p}: \mathscr F(\mathcal{M},\tau) \rightarrow[0,\infty)
\]
is defined by
\[
\|A\|_{p}=(\tau(|A|^{p})^{1/p}, \ \ \forall\ A\in\mathscr F(\mathcal{M}%
,\tau).
\]
It is a highly non-trivial fact that $\|\cdot\|_{p}$ is a norm on
$\mathscr F(\mathcal{M},\tau)$. We let $L^{p}(\mathcal{M},\tau)$ be the
completion of $\mathscr F(\mathcal{M},\tau)$ with respect to the norm
$\|\cdot\|_{p}$ (see \cite{Pisier} for more details). When $p=\infty$, we let
$\|A\|_{\infty}=\|A\|$ for all $A\in\mathcal{M}$ and let $L^{\infty
}(\mathcal{\mathcal{M}},\tau)=\mathcal{M}$.


\subsection{Spectral theory for self-adjoint operators}

Recall a densely defined, closed operator $A$ is affiliated with $\mathcal{M}$
if $AU^{\prime}=U^{\prime}A$ for all unitary operator $U^{\prime}$ in
$\mathcal{M}^{\prime}$, where $\mathcal{M}^{\prime}$ is the commutant of
$\mathcal{M}$ in $\mathcal{B}(\mathcal{H})$. Let $\mathscr A(\mathcal{M})$ be
the set of all densely defined, closed operators that are affiliated with
$\mathcal{M}$. Note that, from Theorem 5 in \cite{Nelson}, $L^{p}%
(\mathcal{M},\tau)$ can be identified as a subset of $\mathscr A(\mathcal{M})$
for each $1\le p\le\infty$.


Let $H$ be a  self-adjoint element in $\mathscr
A(\mathcal{M})$. Then there exists a family $\{E(\lambda
)\}_{\lambda\in\mathbb{R}}$ of projections  in $\mathcal{M}$ that is the spectral
resolution of the identity for $H$ such that
$
H=\int_{-\infty}^{\infty}\lambda dE(\lambda).
$  (In fact, each
$E(\lambda)$ is a spectral projection of $H$ corresponding to the interval
$(-\infty,\lambda]$. See Theorem 5.2.6 of \cite {Kadison} for more details.)
%
If $f$ is a bounded Borel function on $\mathbb{R}$, then $f(H)$ is
an element in the von Neumann subalgebra generated by $\{E(\lambda
)\}_{\lambda\in\mathbb{R}}$ in $\mathcal{M}$, satisfying
\begin{align}
\langle f(H)x,y\rangle=\int_{-\infty}^{\infty}f(\lambda) \ d \langle
E(\lambda) x, y\rangle, \qquad\forall\ x,y\in\mathcal{H}. \label{equ3.1}%
\end{align}
In particular, for each $t\in\mathbb{R}$, we have
\begin{align}
\langle e^{-itH}x,y\rangle=\int_{-\infty}^{\infty}e^{-it\lambda}\ d \langle
E(\lambda) x, y\rangle, \qquad\forall\ x,y\in\mathcal{H}.\label{equ3.2}
\end{align}
%
%



\subsection{Absolutely continuous spectrum}

Let $H$ be a  self-adjoint element in $\mathscr A(\mathcal{M}%
)$ and let $\{E(\lambda)\}_{\lambda\in\mathbb{R}}$ be the spectral resolution
of the identity for $H$ in $\mathcal{M}$.
We let $\mathcal{H}_{ac}(H)$ be the set of all these vectors $x$ in
$\mathcal{H}$ such that the mapping $\lambda\mapsto\langle E(\lambda) x,
x\rangle$, with $\lambda\in\mathbb{R}$, is a (locally) absolutely continuous
function on $\mathbb{R}$ (see \cite{Kato4} for details of the definition). It
is known that $\mathcal{H}_{ac}(H)$ is a closed subspace of $\mathcal{H}$ (see
Theorem X.1.5 of \cite{Kato4}). Let $P_{ac}(H)$ be the projection from
$\mathcal{H}$ onto $\mathcal{H}_{ac}(H)$. Then $P_{ac}(H)$ is in the von
Neumann subalgebra generated by $\{E(\lambda)\}_{\lambda\in\mathbb{R}}$ in
$\mathcal{M}$ (see Theorem X.1.6 in \cite{Kato4}).


%
%
%
%
%

The following result can be found in the proof of Theorem X.4.4 of
\cite{Kato4}.

\begin{lemma}
\label{lemma3.4.2} Let $x\in\mathcal{H}_{ac}(H)$. If $\Delta$ is a Borel
subset of $\mathbb{R}$ and $\chi_{_{\Delta}}$ is the characteristic function
of $\Delta$, then $\chi_{_{\Delta}}(H) x\in\mathcal{H}_{ac}(H)$ and
\[
\frac{d \langle E(\lambda) \ \chi_{_{\Delta}}(H)x, x\rangle}{d\lambda}
=\chi_{_{\Delta}}(\lambda) \frac{d \langle E(\lambda) x, x\rangle}{d\lambda},
\quad \text{ for $\lambda\in\mathbb R $ a.e.}
\]

\end{lemma}


We end this subsection with an example of a self-adjoint operator $A$ in a
semifinite von Neumann algebra $\mathcal{M}$ such that $A$ has purely
absolutely continuous spectrum and $A$ is an arbitrarily small $\max
\{\|\cdot\|, \|\cdot\|_{1}\}$-norm perturbation of a diagonal operator.

\begin{example}\label{example2.4.2}
Let $\mathcal{N}$ be a diffuse finite von Neumann algebra with a faithful
normal tracial state $\tau_{\mathcal{N}}$ and let $\mathcal{H}_{0}$ be an
infinite dimensional separable Hilbert space. Then $\mathcal{M}=\mathcal{N}%
\otimes\mathcal{B}(\mathcal{H}_{0})$ is a semifinite von Neumann algebra with a
faithful normal tracial weight $\tau_{\mathcal{M}}=\tau_{\mathcal{N}} \otimes
Tr$, where $Tr$ is the canonical trace of $\mathcal{B}(\mathcal{H}_{0})$. We
might further assume that $\mathcal{M}$ acts naturally on the Hilbert space
$\mathcal{H}=L^{2}(\mathcal{N},\tau_{\mathcal{N}})\otimes\mathcal{H}_{0}$.

Let $\{E(\lambda)\}_{0\le\lambda\le1}$ be an increasing family of projections
in $\mathcal{N}$ such that $\tau(E(\lambda))=\lambda$ for $0\le\lambda\le1$.
Let $X=\int_{0}^{1} \lambda dE(\lambda)$ and $A=X\otimes I_{\mathcal{B}%
(\mathcal{H}_{0})}$. Notice we can identify $\mathcal{N}$ with a subset of
$L^{2}(\mathcal{N},\tau_{\mathcal{N}})$. For any unitary element $u$ in
$\mathcal{N}\subseteq L^{2}(\mathcal{N},\tau_{\mathcal{N}})$ and any unit
vector $y$ in $\mathcal{H}_{0}$, we have
\[
\langle(E(\lambda)\otimes I_{\mathcal{B}(\mathcal{H})})(u\otimes y), u\otimes
y\rangle= \tau(u^{*}E(\lambda)u)\|y\|^{2} =\lambda, \ \ \text{ for }
0\le\lambda\le1.
\]
Thus the vector $u\otimes y$ is in $\mathcal{H}_{ac}(A)$, which shows that
$\mathcal{H}_{ac}(A) = L^{2}(\mathcal{N},\tau_{\mathcal{N}})\otimes
\mathcal{H}_{0}$.

On the other hand, there exist a family $\{P_{n}\}_{n\in\mathbb{N}}$ of
orthogonal projections in $\mathcal{B}(\mathcal{H}_{0})$ such that
$Tr(P_{n})=1$ for each $n\in\mathbb{N}$ and $\sum_{n} P_{n}=I_{\mathcal{B}%
(\mathcal{H}_{0})}$. Then $A=X\otimes I_{\mathcal{B}(\mathcal{H}_{0})}=
\sum_{n} X\otimes P_{n}$. Let $\epsilon>0$ be given. For each $n\in\mathbb{N}%
$, by spectral theory, there exists a self-adjoint diagonal operator $Y_{n}%
\in\mathcal{N}$ such that $\|X-Y_{n}\|\le\epsilon/2^{n}.$ Let $Y=\sum_{n}
Y_{n}\otimes P_{n}$. Then $Y$ is a self-adjoint diagonal element in
$\mathcal{M}$ such that $\max\{\|A -Y\|, \|A -Y\|_{1} \}\le\epsilon.$

Thus there is a  self-adjoint element $A$ in $\mathcal{M}$ with
purely absolutely continuous spectrum such that $A$ is an
arbitrarily small $\max\{ \|\cdot\|, \|\cdot\|_{1} \}$-norm
perturbation of a diagonal element in $\mathcal{M}$.
\end{example}

\subsection{Identification operator}

For $H\in\mathscr A(\mathcal{M})$, we denote the domain of $H$ in
$\mathcal{H}$ by $\mathcal{D}(H)$.

\begin{lemma}
\label{lem3.5.1} Assume that $H_{1}$ and $H$ are  self-adjoint
elements in $\mathscr A(\mathcal{M})$. Let $J$ be an element in
$\mathcal{M}$ such that $J\mathcal{D}(H)\subseteq\mathcal{D}(H_{1})$
and $H_{1}J-JH$ extends to a bounded operator $B$ in $\mathcal{M}.$
Let
\[
W(t)=e^{itH_{1}}Je^{-itH}, \qquad\text{ for } t \in\mathbb{R}.
\]
Then, for all $x,y\in\mathcal{H}$ and $s,t\in\mathbb{R}, a>0$,

\begin{enumerate}
\item[(i)] the mapping $\lambda\mapsto e^{ i\lambda H_{1}} B e^{-i\lambda H}x$
from $[s,t]$ into $\mathcal{H}$ is Bochner integrable with
\[
\big( W(t)-W(s)\big)x = i \int_{s}^{t} e^{ i\lambda H_{1}} B e^{-i\lambda H}x
d\lambda,
\]

\item[(ii)]
\[
\begin{aligned}
\langle & \left (W(t)^* (W(t)-W(s))- e^{ iaH}W(t)^*(W(t)-W(s))e^{-iaH}  \right )x, y\rangle \\ &=i\int_0^a \langle e^{ i(\lambda+t) H}   \left ( B^*  J- B^* e^{-i(s-t) H_1}J -J^*   B+ J^*e^{-i(t-s)  H_1} B\right )   e^{-i(\lambda+s) H}  x, y\rangle d\lambda .\end{aligned}
\]

\end{enumerate}
\end{lemma}

\begin{proof}
The proof can be found in Chapter X (3.21) and Chapter X (5.8) of \cite{Kato4}
(also see  \cite{Pearson}).

\end{proof}

\begin{proposition}
Assume that $H_{1}$ and $H$ are  self-adjoint elements in $\mathscr
A(\mathcal{M})$. Let $J$ be an element in $\mathcal{M}$ such that
$J\mathcal{D}(H)\subseteq\mathcal{D}(H_{1})$ and $H_{1}J-JH$ extends
to a bounded operator $B$ in $\mathcal{M}\cap
L^{1}(\mathcal{M},\tau).$ Let $W(t)=e^{itH_{1}}Je^{-itH}, \text{ for
} t \in\mathbb{R}. $ Then  for all $t, s\in \mathbb{R}$,
\[
W(t)-W(s) \in\mathcal{M}\cap L^{1}(\mathcal{M},\tau).
\]

\end{proposition}

\begin{proof}
We need only to show that $W(t)-W(s) \in L^{1}(\mathcal{M},\tau). $ By Lemma
\ref{lemma3.1.1}, there exists an increasing sequence $\{P_{n}\}_{n\in
\mathbb{N}}$ of projections in $\mathcal{M}$ such that $\tau(P_{n})<\infty$
for each $n\in\mathbb{N}$ and $P_{n}\rightarrow I$ in strong operator
topology. Fix an $n\in\mathbb{N}$. As the mapping $A\mapsto\tau(AP_{n})$
defines a normal positive linear functional on $\mathcal{M}$, from Theorem
7.1.11 in \cite{Kadison}, there exists an orthogonal family $\{y_{m}%
\}_{m\in\mathbb{N}}$ of vectors in $\mathcal{H}$ such that
\[
\sum_{m}\|y_{m}\|^{2}<\infty\qquad\text{ and }\qquad\tau(AP_{n})=\sum_{m}
\langle Ay_{m}, y_{m}\rangle, \ \forall\ A\in\mathcal{M}.
\]
Let $W(t)-W(s)=W|W(t)-W(s)|$ be the polar decomposition of $W(t)-W(s)$ in
$\mathcal{M}$, where $W$ is a partial isometry and $|W(t)-W(s)|$ is a positive
operator in $\mathcal{M}$. Then, from Lemma \ref{lem3.5.1}, it induces that
\begin{align}
\tau(|W(t)-W(s)|P_{n})  &  = \sum_{m} \langle|W(t)-W(s)| y_{m}, y_{m}\rangle=
\sum_{m} \langle(W(t)-W(s)) y_{m}, Wy_{m}\rangle\nonumber\\
&  = i \sum_{m} \int_{s}^{t} \langle e^{ i\lambda H_{1}} B e^{-i\lambda
H}y_{m}, Wy_{m}\rangle d\lambda.\label{equ3.5}%
\end{align}
Observe that
\begin{align}
\sum_{m} |\langle e^{ i\lambda H_{1}} B e^{-i\lambda H}y_{m}, Wy_{m}\rangle|
\le\sum_{m} \big( \|B\|\|y_{m}\|^{2}\big) <\|B\| \cdot\sum_{m}\|y_{m}%
\|^{2}<\infty. \label{equ3.6}%
\end{align}
Combining (\ref{equ3.5}), (\ref{equ3.6}) and applying the Lebesgue Dominating
Theorem,
\[
\begin{aligned}
\tau(|W(t)-W(s)|P_n)&=   i    \int_s^t \sum_{m}\langle  e^{ i\lambda H_1} B e^{-i\lambda H}y_m, Wy_m\rangle d\lambda \\
&=   i    \int_s^t \sum_{m} \langle  We^{ i\lambda H_1} B e^{-i\lambda H}y_m,  y_m\rangle d\lambda \\
& = i \int_s^t \tau(We^{ i\lambda H_1} B e^{-i\lambda H}P_n) d\lambda.
\end{aligned}
\]
This implies that, for all $n\in\mathbb{N}$,
\[
|\tau(|W(t)-W(s)|P_{n})|\le\int_{s}^{t} \left|  \tau(We^{ i\lambda H_{1}} B
e^{-i\lambda H}P_{n}) \right|  d\lambda\le(t-s) \|B\|_{1}.
\]
Since $\tau$ is a normal weight of $\mathcal{M}$ and $P_{n}\rightarrow I$ in
strong operator topology, we conclude that
\[
\|W(t)-W(s)\|_{1}=\tau(|W(t)-W(s)|)=\sup_{n} \tau(|W(t)-W(s)|P_{n}) \le(t-s)
\|B\|_{1}.
\]
This ends the proof of the proposition.
\end{proof}

\section{$\mathcal{M}$-$H$-smoothness in Semifinite von Neumann Algebras}

Let $\mathcal{H}$ be a complex Hilbert space. Let $\mathcal{M} $ be
a countably decomposable, properly infinite, semifinite von Neumann algebra acting on $\mathcal H$ and   $\tau$
a faithful normal semifinite tracial weight of $\mathcal M$.  Let $\mathscr
A(\mathcal{M})$ be the set of densely defined, closed operators that
are affiliated with $\mathcal{M}$.

\subsection{A smooth condition}

The following definition will be crucial when showing the existence of wave
operators in semifinite von Neumann algebras.

\begin{definition}
\label{def4.2.8} Let $H$ be a  self-adjoint element in
$\mathscr A(\mathcal{M})$. A pair $(A, x)$ is said to be $\mathcal{M}$%
-$H$-smooth if

\begin{enumerate}
\item[(i)] $A\in\mathcal{M}\cap L^{2}(\mathcal{M},\tau)$ and $x\in\mathcal{H}$;

\item[(ii)] there exists a positive constant $c$ such that
\[
\int_{\mathbb{R}} \|ASe^{-i\lambda H} x \|^{2} d\lambda\le c^{2}\|S\|^{2},
\qquad\forall S\in\mathcal{M}.
\]

\end{enumerate}
\end{definition}

\subsection{Point-wise convergence of wave operator}

\begin{proposition}
\label{prop5.0.1} Let $H$ be a  self-adjoint element in $\mathscr
A(\mathcal{M})$. Assume  $B\in\mathcal{M}\cap L^{1}(\mathcal{M},\tau)$
and $x\in\mathcal{H}$ such that $(|B|^{1/2}, x)$ is
$\mathcal{M}$-$H$-smooth. Then there exists a positive constant $c$
such that, for all $s,t\in \mathbb{R}, a >0$ and $S\in\mathcal{M}$,
\[
\begin{aligned}
|\langle   \int_0^a e^{i(\lambda+t)H} SB e^{-i(\lambda+s)H}  x \ d\lambda , x  \rangle|
\le  c \|S\|   \left ( \int_s^\infty \||B|^{1/2} e^{-i\lambda H}x\|^2 d\lambda\right )^{1/2}
\end{aligned}
\]
and
\[
\begin{aligned}
|\langle   \int_0^a e^{i(\lambda+t)H}  B^*S e^{-i(\lambda+s)H} x \ d\lambda   , x  \rangle|
\le  {c} \|S\|  \left ( \int_t^\infty \||B|^{1/2} e^{-i\lambda H}x\|^2 d\lambda\right )^{1/2}.
\end{aligned}
\]

\end{proposition}

\begin{proof}
Let $s,t\in\mathbb{R}, a >0$ be given. As $(|B|^{1/2}, x)$ is $\mathcal{M}%
$-$H$-smooth, there exists a positive constant $c$ such that
\begin{align}
\int_{\mathbb{R}} \||B|^{1/2}Se^{-i\lambda H} x \|^{2} d\lambda\le
c^{2}\|S\|^{2}, \qquad\forall S\in\mathcal{M}. \label{equ4.1}%
\end{align}
Assume that $B=W|B|$ is the polar decomposition of $B$ in $\mathcal{M}$, where
$W$ is a partial isometry in $\mathcal{M}$ and $|B|$ is a positive operator in
$\mathcal{M}$. We have
\begin{align}
&  |\int_{0}^{a} \langle e^{i(\lambda+t)H} SB e^{-i(\lambda+s)H} x, x\rangle
d\lambda|\nonumber\\
&  \qquad\qquad\le\int_{0}^{a} |\langle|B|^{1/2} e^{-i(\lambda+s)H}
x,|B|^{1/2} W^{*} S^{*}e^{-i(\lambda+t)H} x\rangle| d\lambda\nonumber\\
&  \qquad\qquad\le\int_{0}^{a} \Big\| |B|^{1/2} W^{*} S^{*}e^{-i(\lambda+t)H}
x\Big\| \Big\| |B|^{1/2} e^{-i(\lambda+s)H} x \Big\| d\lambda\nonumber\\
&  \qquad\qquad\le\left(  \int_{0}^{a} \Big\||B|^{1/2} W^{*} S^{*}
e^{-i(\lambda+t)H} x\Big\|^{2} d\lambda\right)  ^{1/2} \left(  \int_{0}^{a}
\Big\| |B|^{1/2} e^{-i(\lambda+s)H} x \Big\|^{2} d\lambda\right)
^{1/2}\nonumber\\
&  \qquad\qquad= \left(  \int_{t}^{a+t} \Big\||B|^{1/2} W^{*}S^{*}
e^{-i\lambda H} x\Big\|^{2} d\lambda\right)  ^{1/2} \left(  \int_{s}^{a+s}
\Big\| |B|^{1/2} e^{-i \lambda H} x \Big\|^{2} d\lambda\right)  ^{1/2}%
\nonumber\\
&  \qquad\qquad\le\left(  \int_{\mathbb{R}} \Big\| |B|^{1/2} W^{*}S^{*}
e^{-i\lambda H} x\Big\|^{2} d\lambda\right)  ^{1/2} \left(  \int_{s}^{\infty
}\Big\| |B|^{1/2} e^{-i \lambda H} x \Big\|^{2} d\lambda\right)
^{1/2}\nonumber\\
&  \qquad\qquad\le{c} \|S\| \left(  \int_{s}^{\infty}\Big\| |B|^{1/2} e^{-i
\lambda H} x \Big\|^{2} d\lambda\right)  ^{1/2}. \tag{by (\ref{equ4.1})}%
\end{align}
Similarly, we have
\[
\begin{aligned}
|\langle   \int_0^a e^{i(\lambda+t)H}  B^*S e^{-i(\lambda+s)H} x \ d\lambda   , x  \rangle|
\le  {c} \|S\|  \left ( \int_t^\infty |B|^{1/2} e^{-i\lambda H}x\|^2 d\lambda\right )^{1/2}.
\end{aligned}
\]
\end{proof}

\begin{lemma}
\label{lemma5.1.2} Suppose  $H$ and $H_{1}$ are  self-adjoint elements in
$\mathscr A(\mathcal{M})$. Assume  $J$ is in $\mathcal{M}$ such that
$J\mathcal{D}(H)\subseteq\mathcal{D}(H_{1})$ and the closure of
$H_{1}J-JH$ is in $\mathcal{M}\cap L^{1}(\mathcal{M},\tau)$. Let
$W(t)=e^{ itH_{1}}Je^{-itH}$ for each $t\in\mathbb{R}$.

If $x$ is a vector in $\mathcal{H}$ such that $(|H_{1}J-JH|^{1/2}, x)$ is
$\mathcal{M}$-$H$-smooth, then
\[
\lim_{a\rightarrow\infty} \|(W(t)-W(s))e^{-iaH}x\|=0, \qquad\text{ for all
$t>s$.}
\]

\end{lemma}

\begin{proof}
Note that $(|H_{1}J-JH|^{1/2}, x)$ is $\mathcal{M}$-$H$-smooth. There exists a
positive number $c$ such that
\begin{align}
\int_{\mathbb{R}} \||H_{1}J-JH|^{1/2} e^{-i\lambda H} x \|^{2} d\lambda\le
c^{2} . \label{equ4.2}%
\end{align}
From Lemma \ref{lem3.5.1},
\[
\begin{aligned}
\left \| \big( W(t)-W(s)\big)e^{-iaH}x\right \|  &\le  \int_s^t   \| e^{ i\lambda H_1} (H_1J-JH) e^{-i\lambda H}e^{-iaH}x \| d\lambda \\
&\le \int_s^t   \|  |H_1J-JH|^{1/2}\| \| |H_1J-JH|^{1/2} e^{-i\lambda H}e^{-iaH}x \| d\lambda \\
&\le \sqrt{(t-s)\|  |H_1J-JH|^{1/2}\|^2}\sqrt{ \int_s^t     \| |H_1J-JH|^{1/2} e^{-i(\lambda+a) H}x \|^2 d\lambda}\\
& = \sqrt{(t-s)\|  H_1J-JH\|}\sqrt{ \int_{a+s}^{a+t}     \| |H_1J-JH|^{1/2} e^{-i\lambda  H} x\|^2 d\lambda}\\
& \le \sqrt{(t-s)\|  H_1J-JH\|}\sqrt{ \int_{a+s}^{\infty} \|
|H_1J-JH|^{1/2} e^{-i\lambda H} x\|^2 d\lambda}
\end{aligned}
\]
From (\ref{equ4.2}) we have
\[
\lim_{a\rightarrow\infty} \|(W(t)-W(s))e^{-iaH}x\|=0.
\]
\end{proof}

The proof of the next result follows a strategy by Pearson in \cite{Pearson}.

\begin{proposition}
\label{thm5.0.2} Suppose $H$ and $H_{1}$ are  self-adjoint elements in
$\mathscr A(\mathcal{M})$. Assume $J$ is in $\mathcal{M}$ such that
$J\mathcal{D}(H)\subseteq\mathcal{D}(H_{1})$ and the closure of
$H_{1}J-JH$ is in $\mathcal{M}\cap L^{1}(\mathcal{M} ,\tau)$. Let
$W(t)=e^{ itH_{1}}Je^{-itH}$ for each $t\in\mathbb{R}$.

If $x$ is a vector in $\mathcal{H}$ such that $(|H_{1}J-JH|^{1/2}, x)$ is
$\mathcal{M}$-$H$-smooth, then
$W(t) x$ converges in $\mathcal{H}$ as $t\rightarrow\infty$.

\end{proposition}

\begin{proof}
To prove the result, it suffices to show that for every $\epsilon>0$ there
exists an $N>0$ such that, if $t>s>N$, then $\|(W(t)-W(s))x\|<\epsilon.$

Denote by $B$ the closure of $H_{1}J-JH$. Thus $B$ is in $\mathcal{M}\cap
L^{1}(\mathcal{M},\tau)$. Let $\epsilon>0$ be given. From Lemma \ref{lem3.5.1}%
, for any $t,s, a>0$, we have
\[
\begin{aligned}
\langle & \left (W(t)^* (W(t)-W(s))- e^{ iaH}W(t)^*(W(t)-W(s))e^{-iaH}  \right )x, x\rangle \\ &=i\int_0^a \langle e^{ i(\lambda+t) H}   \left ( B^*  J- B^* e^{-i(s-t) H_1}J -J^*   B+ J^*e^{-i(t-s)  H_1} B)\right )   e^{-i(\lambda+s) H}  x, x\rangle d\lambda
.\end{aligned}
\]
As $(|B|^{1/2}, x)$ is $\mathcal{M}$-$H$-smooth, by Proposition
\ref{prop5.0.1}, there exists an $N_{1}>0$ such that for all $t, s>N_{1}$ and
all $a>0$, we have
\begin{align}
\left|  \langle\left(  W(t)^{*} (W(t)-W(s))- e^{ iaH}W(t)^{*}%
(W(t)-W(s))e^{-iaH} \right)  x, x\rangle\right|  <\frac\epsilon4.
\label{equ4.3}%
\end{align}
For each $t, s>N_{1}$, from Lemma \ref{lemma5.1.2} it follows that%
\begin{align}
\|\langle W(t)^{*}(W(t)-W(s))e^{-iaH} x, x\rangle\|<\frac\epsilon4, \ \text{
when } a \text{ is large enough.} \label{equ4.4}%
\end{align}
Thus, from (\ref{equ4.3}) and (\ref{equ4.4}), we conclude that, for $t,
s>N_{1}$,
\begin{align}
\left|  \langle\left(  W(t)^{*} (W(t)-W(s)) \right)  x, x\rangle\right|
<\frac\epsilon2. \label{equ4.5}%
\end{align}
Similarly, there exists an $N_{2}>0$ such that, when $t, s>N_{2}$, we have
\begin{align}
\left|  \langle\left(  W(s)^{*} (W(t)-W(s)) \right)  x, x\rangle\right|
<\frac\epsilon2. \label{equ4.6}%
\end{align}
Now, (\ref{equ4.5}) and (\ref{equ4.6}) imply that, for all $t,s>\max
\{N_{1},N_{2}\}$,
\[
\left\|  (W(t)-W(s)) x\right\|  <\epsilon,
\]
which ends the proof of the proposition.
\end{proof}

\section{Norm Absolutely Continuous Projections in Semifinite von Neumann
algebras}

Let $\mathcal{H}$ be a complex Hilbert space. Let $\mathcal{M} $ be
a countably decomposable, properly infinite, semifinite von Neumann
algebra acting on $\mathcal H$ and   $\tau$ a faithful normal
semifinite tracial weight of $\mathcal M$.  Let $\mathscr
A(\mathcal{M})$ be the set of densely defined, closed operators
affiliated with $\mathcal{M}$.

\subsection{Norm absolutely continuous projections}

\begin{definition}
\label{def5.1.1} Let $H$ be a  self-adjoint element in $\mathscr
A(\mathcal{M})$ and let $\{E(\lambda)\}_{\lambda\in\mathbb{R}}$ be
the spectral resolution of the identity for $H$ in $\mathcal{M}$. We
define $\mathscr P_{ac}^{\infty}(H)$ to be the collection of those
projections $P$ in $\mathcal{M}$ such that

\begin{enumerate}

\item [] \emph{the mapping $\lambda\mapsto PE(\lambda) P $ from $\lambda
\in\mathbb{R}$ into $\mathcal{M}$ is locally absolutely continuous, i.e. for
all $a,b\in\mathbb{R}$ with $a<b$ and every $\epsilon>0$, there exists a
$\delta>0$ such that $\sum_{i}\|PE({ b_{i}}) P - PE({ a_{i}} )P \|< \epsilon$
for every finite collection $\{(a_{i},b_{i})\}$ of disjoint intervals in
$[a,b]$ with $\sum_{i}(b_{i}-a_{i})<\delta.$}
\end{enumerate}

A projection $P$ in $\mathscr P_{ac}^{\infty}(H)$ is called a norm
absolutely continuous projection with respect to $H$.
\end{definition}

\begin{remark}
Definition \ref{def5.1.1} is closely related to the concept of
$H$-smoothness introduced by Kato in \cite{Kato2} for
$\mathcal{B}(\mathcal{H})$, where $\mathcal{H}$ is a complex
separable Hilbert space. From one of equivalent definitions of
$H$-smoothness in Theorem 5.1 of \cite{Kato2}, it is not hard to see
that if a projection $P$ is $H$-smooth, then $P\in\mathscr
P_{ac}^{\infty}(H)$.
\end{remark}

In the case when $\mathcal{M}=\mathcal{B}(\mathcal{H})$, the
following proposition relates $\mathscr P_{ac}^{\infty}(H)$ to
$\mathcal{H}_{ac}(H)$.

\begin{proposition}
\label{example5.1.3} Let $\mathcal{H}$ be a complex infinite
dimensional separable  Hilbert space and let
$\mathcal{B}(\mathcal{H})$ be the set of all bounded linear
operators on $\mathcal{H}$. Assume $H$ is a densely defined
self-adjoint operator on $\mathcal{H}$ \emph{(}so $H \in\mathscr A(\mathcal{B}%
(\mathcal{H}))$\emph{)}. Then a vector $x$ is in
$\mathcal{H}_{ac}(H)$ if and only if the rank one projection
$x\otimes x$ is in $\mathscr P_{ac}^{\infty }(H)$, where $x\otimes
x$ is defined by $(x\otimes x)y= \langle y,x\rangle x$ for all
$y\in\mathcal{H}$.
\end{proposition}

\begin{proof}
Let $\{E(\lambda)\}_{\lambda\in\mathbb{R}}$ be the spectral resolution of the
identity for $H$ in $\mathcal{B}(\mathcal{H})$. Let $x$ be a vector in
$\mathcal{H}$. For all $\lambda\in\mathbb{R }$, we have
\[
(x\otimes x) E(\lambda) (x\otimes x) = \langle E(\lambda) x,x\rangle(x\otimes
x).
\]
Thus, $x$ is in $\mathcal{H}_{ac}(H)$ if and only if $x\otimes x$ is
in $\mathscr P_{ac}^{\infty}(H)$.
\end{proof}

Next  example shows there exists   self-adjoint operators in $\mathcal M$ with nonzero norm absolutely continuous projections.
\begin{example}
\label{example5.1.4} Let $\mathcal{N}$ be a diffuse finite von Neumann algebra
with a faithful normal tracial state $\tau_{\mathcal{N}}$ and let
$\mathcal{H}_{0}$ be an infinite dimensional separable Hilbert space. Then
$\mathcal{M}=\mathcal{N}\otimes\mathcal{B}(\mathcal{H}_{0})$ is a semifinite
von Neumann algebra with a faithful normal tracial weight $\tau_{\mathcal{M}%
}=\tau_{\mathcal{N}} \otimes Tr$, where $Tr$ is the canonical trace
of $\mathcal{B}(\mathcal{H}_{0})$. We might further assume that
$\mathcal{M}$
acts naturally on the Hilbert space $L^{2}(\mathcal{N},\tau_{\mathcal{N}%
})\otimes\mathcal{H}_{0}$.

Let $X $ be a  densely defined, self-adjoint operator with purely
absolutely continuous spectrum on $\mathcal{H}_{0}$. Then
$I_{\mathcal{N}}\otimes X$ is a densely defined, self-adjoint
operator affiliated with $\mathcal{M}$. For each vector
$y\in\mathcal{H}_{0}$, denote by $Q_{y}$ the rank one projection
$y\otimes y$ in $\mathcal{B}(\mathcal{H}_{0})$. We claim that $I_{\mathcal{N}%
}\otimes Q_{y} \in\mathscr P_{ac}^{\infty}(I_{\mathcal{N}}\otimes X)$. In
fact, if $\{E(\lambda)\}_{\lambda\in\mathbb{R}}$ is the spectral resolution of
the identity for $X$ in $\mathcal{B}(\mathcal{H}_{0})$, then $\{I_{\mathcal{N}%
}\otimes E(\lambda)\}_{\lambda\in\mathbb{R}}$ is the spectral resolution of
the identity for $I_{\mathcal{N}}\otimes X$ in $\mathcal{N}\otimes
\mathcal{B}(\mathcal{H}_{0})=\mathcal{M}$. Note $\|(I_{\mathcal{N}}\otimes
Q_{y}) (I_{\mathcal{N}}\otimes E(\lambda)-I_{\mathcal{N}}\otimes
E(\mu))(I_{\mathcal{N}}\otimes Q_{y})\|=\|Q_{y}(E(\lambda)- E(\mu
))Q_{y}\|=\langle(E(\lambda)- E(\mu))y, y \rangle\|y\|^{2} $ for all
$\lambda>\mu$. We have that $I_{\mathcal{N}}\otimes Q_{y} \in\mathscr P_{ac}%
^{\infty}(I\otimes X)$.
\end{example}

It is easy to see the following statement.

\begin{lemma}
\label{lemma4.1.2} Suppose  $H$ is a  self-adjoint element in
$\mathscr A(\mathcal{M})$. If $P\in\mathscr P_{ac}^{\infty}(H)$,
then $P\le P_{ac}(H)$, where $P_{ac}(H)$ is the projection from
$\mathcal{H}$ onto $\mathcal{H}_{ac}(H)$.
\end{lemma}

\begin{definition}
\label{def5.1.5} Suppose that $P\in\mathscr P_{ac}^{\infty}(H)$. For
each interval $[a,b]$, we define
\[
V_{[a,b]} (P) = \sup\{ \sum_{i=1}^{m} \|PE_{\lambda_{i}}P-PE_{\lambda_{i-1}%
}P\| : a=\lambda_{0}<\lambda_{1}<\cdots<\lambda_{m}=b \text{ is a
partition of } [a,b] \}
\]
and
\[
\Psi(\lambda)= \left\{
\begin{aligned} & V_{[0,\lambda]} (P) \qquad \text{ if } \lambda \ge 0\\
& V_{[ \lambda, 0]} (P) \qquad \text{ if } \lambda < 0 \end{aligned}\right.
\]

\end{definition}


\begin{lemma}
$\Psi$ is locally absolutely continuous on $\mathbb{R}$ and $\Psi^{\prime}$
exists almost everywhere.
\end{lemma}

\begin{proof}It can be verified directly (also see Proposition 1.2.1 in \cite{Arendt}).
\end{proof}

\subsection{Cut-off function $\omega_{n}$}

\begin{definition}
\label{def4.1.5} Suppose $H$ is a  self-adjoint element in $\mathscr
A(\mathcal{M})$. Assume that $P\in\mathscr P_{ac}^{\infty}(H)$ and
$\Psi$ are as in Definition \ref{def5.1.1} and Definition
\ref{def5.1.5}. For each $n\in\mathbb{N}$, we define,
\[
\omega_{n}(\lambda)= \left\{  \begin{aligned}
& 1 \qquad \text { if } |\Psi '(\lambda)|\le n \text { and }   | \lambda| \le n\\
& 0 \qquad \text { otherwise}
\end{aligned}\right.
\]

\end{definition}


\begin{lemma}
\label{lemma5.2.2.1} Suppose  $H$ is a  self-adjoint element in
$\mathscr A(\mathcal{M})$. Assume $P\in\mathscr P_{ac}^{\infty}(H)$
and $x\in\mathcal{H}$. For each $n\in\mathbb{N}$, let
$\omega_{n}(\lambda)$ be as
in Definition \ref{def4.1.5}. Let%
\[
\omega_{n}(H)=\int_{\mathbb{R}} \omega_{n}(\lambda) dE(\lambda).
\]
Then
\[
\omega_{n}(H)\rightarrow I \text{ in strong operator topology, as $n\rightarrow \infty$.}
\]

\end{lemma}

\begin{proof}
Let $\Delta_{n}=\{\lambda\in\mathbb{R }\ : \ \omega_{n}(\lambda)=1\}$. Observe
that $\mathbb{R}\setminus(\cup_{n}\Delta_{n})$ is a measure zero set. Thus
$\omega_{n}(H)\rightarrow I \text{ in strong operator topology, as
$n\rightarrow\infty$.} $
\end{proof}


\begin{lemma}
\label{lemma4.2.3} Suppose $H$ is a  self-adjoint element in
$\mathscr A(\mathcal{M})$. Assume $P\in\mathscr P_{ac}^{\infty}(H)$
and $x\in\mathcal{H}$. For each $n\in\mathbb{N}$, let
$\omega_{n}(\lambda)$ be as in Definition \ref{def4.1.5}. The
following statements are true.

\begin{enumerate}
\item[(i)] For real numbers $a,b$ with $a<b$, the mapping $\lambda\mapsto P
E(\lambda)\omega_{n}(H) x $ from $[a,b]$ into $\mathcal{H}$ is absolutely continuous.

\item[(ii)] We have
\[
\frac{d \big(  P E(\lambda) \omega_{ n}(H) x \big)}{d\lambda} \text{
exists in  $ \mathcal{H}$ for   $\lambda \in \mathbb R$ a.e.}
\]
and the mapping
\[
\lambda\mapsto\frac{d \big(  P E(\lambda) \omega_{ n}(H) x \big)}{d\lambda}
\text{ \ from $\mathbb{R}$ into $\mathcal{H}$ is locally Bochner integrable.}
\]

\item[(iii)] We have
\[
\begin{aligned}
\left\|\frac {d \big(  P E(\lambda) \omega_{ n}(H) x \big)}{d\lambda } \right\|
& \le  \omega_{ n}(\lambda) \sqrt {n} \cdot  \sqrt{\frac {d\langle
E(\lambda)  P_{ac}(H)x,  P_{ac}(H)x  \rangle }{d\lambda}}  \quad
\text{for   $\lambda \in \mathbb R$ a.e.}\end{aligned}
\]

\item[(iv)] We have
\[
\int_{\mathbb{R}} \left\|\frac{d \big( P E(\lambda) \omega_{ n}(H) x \big)}%
{d\lambda} \right\| d\lambda<\infty\quad\text{ and } \quad\int_{\mathbb{R}}
{\left\|\frac{d \big( P E(\lambda) \omega_{ n}(H) x \big)}{d\lambda} \right\|}^{2}
d\lambda<n\|x\|^{2}.
\]

\item[(v)] The mapping
\[
\lambda\mapsto e^{-it\lambda} \frac{d \big( PE(\lambda)\omega_{ n}(H) x
\big)}{d\lambda}%
\]
is in $L^{1}(\mathbb{R}, \mathcal{H})\cap L^{2}(\mathbb{R}, \mathcal{H}) $ for
each $t\in\mathbb{R}$, with
\[
Pe^{-itH}\omega_{{n}}(H) x = \int_{\mathbb{R}} e^{-it\lambda} \frac{d
\big( PE(\lambda)\omega_{ n}(H) x \big)}{d\lambda} d\lambda.
\]

\item[(vi)] We have
\[
\int_{\mathbb{R}} \|Pe^{-i\lambda H}\omega_{{n}}(H) x\|^{2} d\lambda=\frac1
{2\pi} \int_{\mathbb{R}} \left\|  \frac{d \big( PE(\lambda)\omega_{n}(H) x
\big)}{d\lambda}\right\|  ^{2} d\lambda\le\frac{n } {2\pi} \|x\|^{2}.
\]

\end{enumerate}
\end{lemma}

\begin{proof}
(i) \ Recall $P_{ac}(H)$ is the projection from $\mathcal{H}$ onto
$\mathcal{H}_{ac}(H)$. Notice that $P_{ac}(H)$ commutes with $E(\lambda)$ .
Moreover, from Lemma \ref{lemma4.1.2}, $P=PP_{ac}(H)$. Assume that
$\{(a_{i},b_{i})\}$ is a finite family of disjoint intervals in $[a,b]$. Then
\[
\begin{aligned}
&\sum_i \| P  (E(b_i)-E(a_i))\omega_{n}(H) x \| =\sum_i \| P  (E(b_i)-E(a_i))P_{ac}(H) \omega_{n}(H) x  \|  \\ &\qquad    \le\sum_i\|P (E(b_i)-E(a_i))\|  \|  (E(b_i)-E(a_i))P_{ac}(H) \omega_{ n}(H) x\|\\
&\qquad     \le \left (\sum_i \|P (E(b_i)-E(a_i))  \|^2 \right )^{1/2} \left (\sum_i\| (E(b_i)-E(a_i))P_{ac}(H) \omega_{ n}(H) x\|^2 \right )^{1/2}\\ &\qquad     =  \left (\sum_i\|P(E(b_i)-E(a_i))  P\|  \right )^{1/2}    \\ & \qquad\qquad  \qquad  \cdot \left (\sum_i \langle (E(b_i)-E(a_i))\big( P_{ac}(H)\omega_{ n}(H) x\big), \big( P_{ac}(H)\omega_{ n}(H) x\big ) \rangle       \right )^{1/2}.
\end{aligned}
\]
Now the result follows from the fact that $P_{ac}(H)\omega_{ n}(H)
x\in\mathcal{H}_{ac}(H)$ and $P\in\mathscr P_{ac}^{\infty}(H)$.

(ii) The first statement follows from (i) and the Radon-Nikodym
Property of the Hilbert space $\mathcal{H}$ (see Definition 1.2.5 in
\cite{Arendt}). The second statement follows (i) and the first
statement (see Proposition 1.2.3 in \cite{Arendt}).

(iii) We have, from (i),
\[
\begin{aligned}
\left \|\frac {     P (E(\lambda)-E(\mu)) \omega_n(H)x }{ \lambda-\mu }\right \| & \le \frac{ \|(E(\lambda)-E(\mu)) P_{ac}(H) \omega_n(H)x  \| \|P(E(\lambda)-E(\mu))   \|}{| \lambda-\mu| }\\
& = \frac{  \left \|(E(\lambda)-E(\mu)) P_{ac}(H) \omega_n(H)x \right  \| }  {  \sqrt{|\lambda-\mu |} }  \left \| \frac{ P (E(\lambda)-E(\mu))  P }{  \lambda-\mu  }\right \|^{1/2}\\
&\le \sqrt{\left  | \frac{   \langle (E(\lambda)-E(\mu)) \omega_{ n}(H)P_{ac}(H) x, P_{ac}(H) x\rangle }  {  \lambda -\mu} \right  |}  \left  | \frac{ \Psi(\lambda)-\Psi(\mu) }{  \lambda-\mu  }\right  |^{1/2}.
\end{aligned}
\]
Hence, by the definitions of $\omega_{n}$ and $\Psi$, we obtain
\begin{align}
\left\|  \frac{d \big( PE(\lambda)\omega_{n}(H)x \big)}{d\lambda}\right\|   &
\le\sqrt{\left|  \frac{ d \langle E(\lambda) \omega_{ n}(H)P_{ac}(H) x,
P_{ac}(H) x\rangle} { d\lambda} \right|  }\sqrt{\left|  \Psi^{\prime}(\lambda)
\right|  } \quad a.e.\nonumber\\
&  \le\omega_{ n}(\lambda) \sqrt{\left|  \frac{ d \langle E(\lambda)
P_{ac}(H)x, P_{ac}(H) x\rangle} { d\lambda} \right|  }\sqrt{n} \quad a.e.
\tag{by Lemma \ref{lemma3.4.2}}%
\end{align}

(iv) We have
\begin{align}
\int_{\mathbb{R}} \left\|  \frac{d \big( PE(\lambda)\omega_{n}(H) x
\big)}{d\lambda}\right\|  ^{2} d\lambda &  \le\int_{\mathbb{R}} \omega_{
n}(\lambda) \cdot{\frac{d\langle E(\lambda) P_{ac}(H)x, P_{ac}(H)x \rangle
}{d\lambda}} \cdot{n} \ d\lambda\tag{by (iii)}\\
&  \le{n } \|P_{ac}(H)x\|^{2} \ \le{n } \| x\|^{2}. \tag{by (\ref{equ3.1})}%
\end{align}
Similarly,
\begin{align}
\int_{\mathbb{R}} \left\|  \frac{d \big( PE(\lambda)\omega_{n}(H) x
\big)}{d\lambda}\right\|  d\lambda &  \le\int_{\mathbb{R}} \omega_{ n}%
(\lambda)\sqrt{\left|  \frac{ d \langle E(\lambda) P_{ac}(H)x, P_{ac}(H)
x\rangle} { d\lambda} \right|  }\sqrt{n} \ d\lambda\tag{by (iii)}\\
&  \le\left(  \int_{\mathbb{R}} \omega_{ n}(\lambda) \cdot{n} \ d\lambda
\right)  ^{1/2} \left(  \int_{\mathbb{R}} {\frac{d\langle E(\lambda)
P_{ac}(H)x, P_{ac}(H)x \rangle}{d\lambda}} \ d\lambda\right)  ^{1/2}%
\nonumber\\
&  <\infty\nonumber
\end{align}

(v) It follows from (i), (ii) and (iv) that the mapping
\[
\lambda\mapsto e^{-it\lambda} \frac{d \big( PE(\lambda)\omega_{ n}(H) x
\big)}{d\lambda}%
\]
is in $L^{1}(\mathbb{R}, \mathcal{H})\cap L^{2}(\mathbb{R}, \mathcal{H}) $ for
each $t\in\mathbb{R}$. Obviously, $Pe^{-itH}\omega_{{n}}(H) x \in\mathcal{H}.$
We need only to verify that, for all $y\in\mathcal{H}$,
\[
\langle Pe^{-itH}\omega_{{n}}(H) x, y\rangle= \langle\left(  \int_{\mathbb{R}}
e^{-it\lambda} \frac{d \big( PE(\lambda)\omega_{ n}(H) x \big)}{d\lambda}
d\lambda\right)  , y\rangle.
\]
In fact,
\[
\begin{aligned}
\langle  \left ( \int_{\mathbb R} e^{-it\lambda} \frac {d \big( PE(\lambda)\omega_{n}(H) x \big)}{d\lambda } d\lambda \right ), y\rangle
&  =  \int_{\mathbb R} \langle  \left  (e^{-it\lambda} \frac {d \big( PE(\lambda)\omega_{n}(H) x \big)}{d\lambda } \right ), y\rangle d\lambda
\\
&  =  \int_{\mathbb R} e^{-it\lambda} \frac {d  \langle   PE(\lambda)\omega_{n}(H) x , y\rangle  }{d\lambda } d\lambda \\
&  =  \int_{\mathbb R} e^{-it\lambda} \frac {d  \langle     E(\lambda)\omega_{n}(H) x ,  Py\rangle  }{d\lambda } d\lambda \\
& =\langle e^{-itH}\omega_{{n}}(H) x, Py\rangle\\
&  =\langle Pe^{-itH}\omega_{{n}}(H) x, y\rangle.
\end{aligned}
\]
Hence the result holds.

(vi) By Plancherel's Theorem for Fourier transformation on Hilbert spaces (see
Theorem 1.8.1 in \cite{Arendt}), we have
\begin{align}
\int_{\mathbb{R}} \|Pe^{-i\lambda H}\omega_{{n}}(H) x\|^{2} d\lambda &
=\frac1 {2\pi} \int_{\mathbb{R}} \left\|  \frac{d \big( PE(\lambda)\omega
_{n}(H) x \big)}{d\lambda}\right\|  ^{2} d\lambda\le\frac{n } {2\pi} \|
x\|^{2}. \tag{by (v) and (iv)}%
\end{align}

\end{proof}

\begin{proposition}
\label{proposition4.2.6} Suppose  $H$ is a  self-adjoint element in
$\mathscr A(\mathcal{M})$. Assume $P\in\mathscr P_{ac}^{\infty}(H)$
and let $\omega_{n}(\lambda)$ be as in Definition \ref{def4.1.5} for
each $n\in\mathbb{N}$.

If $A\in\mathcal{M}\cap L^{2}(\mathcal{M},\tau)$, then
\[
\int_{\mathbb{R}} \|Ae^{-i\lambda H}\omega_{{n}}(H) P\|_{2}^{2} d\lambda
\le\frac{n } {2\pi} \|A\|_{2}^{2}.
\]

\end{proposition}

\begin{proof}
Note that $e^{-i\lambda H}$ commutes with $\omega_{{n}}%
(H)$. As
\[
\|X\|_{2}^{2}=\tau(X^{*}X)=\tau(XX^{*})=\|X^{*}\|_{2}^{2}, \quad\qquad
\forall\ X\in\mathcal{M},
\]
it suffices to show that
\[
\int_{\mathbb{R}} \|Pe^{-i\lambda H}\omega_{{n}}(H) A\|_{2}^{2} \ d\lambda
\le\frac{n } {2\pi} \|A\|_{2}^{2}.
\]
By Lemma \ref{lemma3.1.1}, there exists a sequence $\{x_{m}\}_{m\in\mathbb{N}%
}$ of vectors in $\mathcal{H}$ such that
\begin{equation}
\|X\|_{2}^{2}=\tau(X^{*}X)= \sum_{m} \langle X^{*}X x_{m}, x_{m}\rangle=
\sum_{m} \|Xx_{m}\|^{2},\qquad\forall\ X\in\mathcal{M}. \label{equ5.1}%
\end{equation}
By Lemma \ref{lemma4.2.3} (vi), for all $m\in\mathbb{N}$,
\begin{equation}
\int_{\mathbb{R}} \|Pe^{-i\lambda H}\omega_{{n}}(H)A x_{m}\|^{2} d\lambda
\le\frac{n } {2\pi} \|Ax_{m}\|^{2}. \label{equ5.2}%
\end{equation}
Thus
\begin{align}
\int_{\mathbb{R}} \|Pe^{-i\lambda H}\omega_{{n}}(H) A\|_{2}^{2} \ d\lambda &
= \int_{\mathbb{R}} \sum_{m} \|Pe^{-i\lambda H}\omega_{{n}}(H) Ax_{m}\|^{2}
\ d\lambda\tag{by (\ref{equ5.1})}\\
&  = \sum_{m} \int_{\mathbb{R}} \|Pe^{-i\lambda H}\omega_{{n}}(H) Ax_{m}\|^{2}
\ d\lambda\nonumber\\
&  \le\sum_{m} \frac{n } {2\pi} \|Ax_{m}\|^{2}\tag{by (\ref{equ5.2})}\\
&  = \frac{n } {2\pi} \|A\|_{2}^{2}. \tag{by (\ref{equ5.1})}%
\end{align}
This ends the proof of the lemma.
\end{proof}


\begin{corollary}
\label{cor4.2.7} Suppose  $H$ is a  self-adjoint element in
$\mathscr A(\mathcal{M})$. Let $P\in\mathscr P_{ac}^{\infty}(H)$ and
let $\omega_{n}(\lambda)$ be as in Definition \ref{def4.1.5} for
each $n\in\mathbb{N}$. Assume $\{x_{m}\}_{m\in\mathbb{N}}$ is a
family of vectors in $\mathcal{H}$ such that
\[
\|X\|_{2}^{2}=\tau(X^{*}X)= \sum_{n} \langle X^{*}X x_{m}, x_{m}\rangle=
\sum_{m} \|Xx_{m}\|^{2},\qquad\forall\ X\in\mathcal{M}.
\]
Then, for all $A\in\mathcal{M}\cap L^{2}(\mathcal{M},\tau)$, $S\in\mathcal{M}$
and $m, n\in\mathbb{N}$, we have
\[
\int_{\mathbb{R}} \|ASe^{-i\lambda H}\omega_{{n}}(H) Px_{m}\|^{2} d\lambda
\le\frac{n } {2\pi} \|A\|_{2}^{2} \|S\|^{2}.
\]

\end{corollary}

\begin{proof}
It follows from Proposition \ref{proposition4.2.6} and the choice of
$\{x_{m}\}_{m\in\mathbb{N}}$ that
\[
\int_{\mathbb{R}} \|ASe^{-i\lambda H}\omega_{{n}}(H) Px_{m}\|^{2} d\lambda
\le\int_{\mathbb{R}} \|ASe^{-i\lambda H}\omega_{{n}}(H) P\|_{2}^{2}
d\lambda\le\frac{n } {2\pi} \|AS\|_{2}^{2} \le\frac{n } {2\pi} \|A\|_{2}^{2}
\|S\|^{2}.
\]

\end{proof}

Now we are ready to state the main result of this section.

\begin{proposition}
\label{prop4.3.2} Suppose  $H$ is a  self-adjoint element in
$\mathscr A(\mathcal{M})$. Let $P\in\mathscr P_{ac}^{\infty}(H)$.
Assume $\{x_{m}\}_{m\in\mathbb{N}}$ is a family of vectors in
$\mathcal{H}$ such that
\[
\|X\|_{2}^{2}=\tau(X^{*}X)= \sum_{m} \langle X^{*}X x_{m}, x_{m}\rangle=
\sum_{m} \|Xx_{m}\|^{2},\qquad\forall\ X\in\mathcal{M}.
\]
Then there exists an increasing sequence $\{Q_{n}\}_{n\in\mathbb{N}}$ of
projections in $\mathcal{M}$ such that

\begin{enumerate}
\item[(i)] $Q_{n}$ converges to $I$ in strong operator topology as $n\rightarrow\infty$.

\item[(ii)] If $A\in\mathcal{M}\cap L^{2}(\mathcal{M},\tau)$ and $m, n
\in\mathbb{N}$, then $(A, Q_{n}Px_{m})$ is $\mathcal{M}$-$H$-smooth.
\end{enumerate}
\end{proposition}

\begin{proof}
Let $Q_{n}=\omega_{n}(H)$ be as in Definition \ref{def4.1.5} for each
$n\in\mathbb{N}$. Now the result follows from Lemma \ref{lemma5.2.2.1},
Corollary \ref{cor4.2.7} and Definition \ref{def4.2.8}.
\end{proof}

\section{Existence of Generalized Wave Operator in Semifinite von Neumann Algebras}

Let $\mathcal{H}$ be a complex Hilbert space and let $\mathcal{B}%
(\mathcal{H})$ be the set of all bounded linear operators on
$\mathcal{H}$. Let $\mathcal{M}\subseteq\mathcal{B}(\mathcal{H})$ be
a countably decomposable, properly infinite von Neumann algebra with
a faithful normal semifinite tracial weight $\tau$. Let $\mathscr
A(\mathcal{M})$ be the set of densely defined, closed operators that
are affiliated with $\mathcal{M}$. Let $\mathscr P_{ac}^{\infty}(H)$
be the set of norm absolutely continuous projections with respect to
$H$ in $\mathcal{M}$.

\subsection{Generalized  wave operators}

\begin{theorem}
\label{thm5.0.3} Suppose $H$ and $H_{1}$ are  self-adjoint elements in
$\mathscr A(\mathcal{M})$. Assume $J$ is in $\mathcal{M}$ such that
$J\mathcal{D}(H)\subseteq\mathcal{D}(H_{1})$ and the closure of
$H_{1}J-JH$ is in $\mathcal{M}\cap L^{1}(\mathcal{M},\tau)$. Let
$W(t)=e^{ itH_{1}}Je^{-itH}$ for each $t\in\mathbb{R}$.

If $P\in\mathscr P_{ac}^{\infty}(H)$, then
$W(t) P$ converges in strong operator
topology in $\mathcal{M}$ as $t\rightarrow\infty$.

\end{theorem}

\begin{proof}
By Lemma \ref{lemma3.1.1}, there exists a family $\{x_{m}\}_{m\in\mathbb{N}}$
of vectors in $\mathcal{H}$ such that
\[
\|X\|_{2}^{2}=\tau(X^{*}X)= \sum_{m} \langle X^{*}X x_{m}, x_{m}\rangle=
\sum_{m} \|Xx_{m}\|^{2},\qquad\forall\ X\in\mathcal{M}.
\]
Note $P\in\mathscr P_{ac}^{\infty}(H)$. By Proposition
\ref{prop4.3.2}, there exists an increasing sequence
$\{Q_{n}\}_{n\in\mathbb{N}}$ of projections in $\mathcal{M}$ such
that (a) $Q_{n}$ converges to $I$ in strong operator
topology as $n\rightarrow \infty$;
and (b) $(A, Q_{n}Px_{m})$ is $\mathcal{M}$-$H$-smooth for all
$A\in\mathcal{M}\cap L^{2}(\mathcal{M},\tau)$ and $m, n
\in\mathbb{N}$. In particular, $(|H_{1}J-JH|^{1/2}, Q_{n}Px_{m})$ is
$\mathcal{M}$-$H$-smooth for all $m, n\in\mathbb{N}$. By Theorem
\ref{thm5.0.2}, $W(t) Q_{n}Px_{m}$ converges in $\mathcal{H}$ as
$t\rightarrow\infty$. Since $Q_{n}$ converges to $I$ in strong operator
topology and
$W(t)$ is uniformly bounded, $W(t) Px_{m}$ converges in
$\mathcal{H}$ as $t\rightarrow\infty$. This further implies that
$W(t) PA^{\prime}x_{m} = A^{\prime} W(t) Px_{m}$ converges in
$\mathcal{H}$, as $t\rightarrow\infty$, for all
$A^{\prime}\in\mathcal{M}^{\prime}$ and $m\in\mathbb{N}$. By Lemma
\ref{lemma3.1.1}, $W(t) Px $ converges in $\mathcal{H}$ for all
$x\in \mathcal{H}$, whence $W(t) P$ converges in strong operator
topology in
$\mathcal{M}$ as $t\rightarrow\infty$.
\end{proof}

\begin{proposition}
\label{thm5.0.4} Let $H$ and $H_{1}$ be  self-adjoint elements in
$\mathscr A(\mathcal{M})$. Suppose that $H_{1} - H$ is in
$\mathcal{M}\cap L^{1}(\mathcal{M},\tau)$. If $P\in\mathscr
P_{ac}^{\infty}(H)$, then
$e^{itH_{1}}e^{-itH} P$ converges in strong operator
topology in $\mathcal{M}$ as $t\rightarrow
\infty$.

\end{proposition}

\begin{proof}
The result is a special case of Theorem \ref{thm5.0.3} when $J=I$.
\end{proof}

\subsection{Kato-Rosenblum Theorem in semifinite von Neumann algebras}

Recall $\mathscr P_{ac}^{\infty}(H)$ is the set of norm absolutely continuous
projections with respect to $H$ in $\mathcal{M}$ (see Definition
\ref{def5.1.1}).

\begin{definition}\label{def5.2.1}
\label{def6.2.1} Suppose $H$ is a  self-adjoint element in $\mathscr
A(\mathcal{M})$. Define
\[
P_{ac}^{\infty}(H) = \vee\{P : P\in\mathscr P_{ac}^{\infty}(H) \}.
\]
Such $P_{ac}^{\infty}(H)$ is called the {\em norm absolutely
continuous support} of $H$ in $\mathcal M$.
\end{definition}

\begin{lemma}\label{lemma5.2.2.1-28}
 Suppose $H$ is a  self-adjoint element in $\mathscr
A(\mathcal{M})$. Then $P_{ac}^{\infty}(H) \le  P_{ac} (H)$. Furthermore, if $\mathcal H$ is   separable  and $\mathcal M=\mathcal B(\mathcal H)$, then $P_{ac}^{\infty}(H) =  P_{ac} (H)$.
\end{lemma}
\begin{proof}
The first statement follows from Lemma \ref{lemma4.1.2} and Definition \ref{def5.2.1}. The second statement follows from a combination of the first statement and  Proposition \ref{example5.1.3}.
\end{proof}
\begin{lemma}
\label{lemma6.2.2} Suppose $H$ is a  self-adjoint element in $\mathscr
A(\mathcal{M})$. Let $\{E(\lambda)\}_{\lambda\in\mathbb{R}}$ be the
spectral resolution of the identity for $H$ in $\mathcal{M}$. If
$S\in\mathcal{M}$ satisfies that the mapping $\lambda\mapsto
S^{*}E(\lambda)S$ from $\mathbb{R}$ into $\mathcal{M}$ is locally
absolutely continuous, then $R(S)$, the range projection of $S$ in
$\mathcal{M}$, is a subprojection of $P_{ac}^{\infty}(H)$.
\end{lemma}

\begin{proof}
Let $S=|S^{*}|W$ be a polar decomposition of $S$ in $\mathcal{M}$ where $W$ is
a partial isometry and $|S^{*}|$ is a positive operator in $\mathcal{M}$. For
each $n\in\mathbb{N}$, let $f_{n}:\mathbb{R}\rightarrow\mathbb{R}$ be a
function such that $f_{n}(\lambda)=1/\lambda$ when $1/n<\lambda<n$ and $0$
otherwise. It is not hard to check that $|S^{*}|\cdot f_{n}(|S^{*}|)$ is a
projection in $\mathcal{M}$ satisfying the mapping $\lambda\mapsto|S^{*}%
|f_{n}(|S^{*}|) E(\lambda)|S^{*}|f_{n}(|S^{*}|) = f_{n}(|S^{*}|)WS^{*}%
E(\lambda)SW^{*}f_{n}(|S^{*}|)$ from $\mathbb{R}$ into $\mathcal{M}$ is
locally absolutely continuous. Therefore, $|S^{*}|f_{n}(|S^{*}|)\in
\mathscr P_{ac}(H)$ for each $n\in\mathbb{N}$. Notice, when $n\rightarrow
\infty$, $|S^{*}|\cdot f_{n}(|S^{*}|)\rightarrow R(S)$ in strong operator topology. Now we conclude
that $R(S)$ is a subprojection of $P_{ac}^{\infty}(H)$.
\end{proof}

\begin{proposition}
\label{prop6.2.3} Suppose $H$ is a  self-adjoint element in $\mathscr
A(\mathcal{M})$. If $\{E(\lambda)\}_{\lambda\in\mathbb{R}}$ is the
spectral resolution of the identity for $H$ in $\mathcal{M}$ and
$\mathcal{A}$ is the von Neumann subalgebra generated by
$\{E(\lambda)\}_{\lambda \in\mathbb{R}}$ in $\mathcal{M}$, then
$P_{ac}^{\infty}(H)$ is in $\mathcal{A}^{\prime}\cap\mathcal{M}$.
\end{proposition}

\begin{proof}
Let $P\in\mathscr P_{ac}^{\infty}(H)$. Let $a\in\mathbb{R}$. It is not hard to verify
that the mapping $\lambda\mapsto PE(a) E(\lambda) E(a)P$ from $\mathbb{R}$
into $\mathcal{M}$ is locally absolutely continuous. Hence, from Lemma
\ref{lemma6.2.2}, $R(E(a)P)$, the range projection of $E(a)P$ in $\mathcal{M}%
$, is a subprojection of $P_{ac}^{\infty}(H)$. It follows that $E(a)P =
P_{ac}^{\infty}(H)E(a)P $. This implies that $E(a)P_{ac}^{\infty}(H)=
P_{ac}^{\infty}(H)E(a)P_{ac}^{\infty}(H)$, whence $E(a)P_{ac}^{\infty}(H)=
P_{ac}^{\infty}(H)E(a)$. Therefore $P_{ac}^{\infty}(H)$ is in $\mathcal{A}%
^{\prime}\cap\mathcal{M}$.
%
\end{proof}

The following is an analogue of Kato-Rosenblum Theorem in semifinite von
Neumann algebras.

\begin{theorem}
\label{thm5.2.2} Suppose $H$ and $H_{1}$ are  self-adjoint elements in
$\mathscr A(\mathcal{M})$ such that $H_{1} - H$ is in
$\mathcal{M}\cap L^{1}(\mathcal{M},\tau)$. Then
\[
W_{+}\triangleq sot\text{-}\lim_{t\rightarrow\infty}e^{ itH_{1}}e^{-itH}
P_{ac}^{\infty}(H) \ \ \text{ exists in } \ \mathcal{M}.
\]
Moreover, \begin{enumerate}
\item [(i)]$\displaystyle
W_{+}^{*}W_{+}= P_{ac}^{\infty}(H)\le P_{ac} (H)$ and $ \displaystyle W_{+} W_{+}^{*}= P_{ac}^{\infty}(H_{1}) \le P_{ac} (H_{1});
$
\item [(ii)]
$\displaystyle
W_+   H W_+^*=   H_1 P_{ac}^{\infty}(H_{1})  .
$
\end{enumerate}
\end{theorem}

\begin{proof}
From Theorem \ref{thm5.0.3} and Definition \ref{def6.2.1}, it follows that
\begin{align}
W_{+}\triangleq sot\text{-}\lim_{t\rightarrow\infty}e^{ itH_{1}}e^{-itH}
P_{ac}^{\infty}(H) \ \ \text{ exists in } \ \mathcal{M}. \label{equ6.1}%
\end{align}
And it is not hard to check that $W_{+}$ is a partial isometry in
$\mathcal{M}$. Therefore
\begin{align}
W_{+}^{*}W_{+}= P_{ac}^{\infty}(H)\le P_{ac} (H). \label{equ6.2}%
\end{align}

Moreover, from Proposition \ref{prop6.2.3}, we have
\begin{align}
W_{+}e^{-isH}  &  =\left(  sot\text{-}\lim_{t\rightarrow\infty}e^{ itH_{1}%
}e^{-itH} P_{ac}^{\infty}(H)\right)  e^{-isH}\tag{by (\ref{equ6.1})}\\
&  = sot\text{-}\lim_{t\rightarrow\infty} \left(  e^{ itH_{1}}e^{-itH} e^{-isH}
P_{ac}^{\infty}(H) \right) \tag{by Proposition \ref{prop6.2.3}}\\
&  = e^{-isH_{1}} \left(  sot\text{-}\lim_{t\rightarrow\infty}e^{ i(t+s)H_{1}%
}e^{-i(t+s)H} P_{ac}^{\infty}(H)\right) \nonumber\\
&  = e^{-isH_{1}}W_{+}, \qquad\forall\ s\in\mathbb{R}. \label{equ6.3}%
\end{align}
Let $\{E(\lambda)\}_{\lambda\in\mathbb{R}}$ and
$\{F(\lambda)\}_{\lambda \in\mathbb{R}}$ be the spectral resolutions
of the identity for $H$  and $H_{1}$ respectively, in $\mathcal{M}$.
Then, for all $x, y\in\mathcal{H}$ and $s\in\mathbb{R}$,
\begin{align}
\langle W_{+}e^{-isH}x, y\rangle &  = \langle e^{-isH}x,W_{+}^{*} y\rangle=
\langle e^{-isH_{1}}W_{+}x, y\rangle\tag{by (\ref{equ6.3})}\\
&  = \int_{\mathbb{R}} e^{-is\lambda} d \langle E(\lambda) x, W_{+}^{*}
y\rangle= \int_{\mathbb{R}} e^{-is\lambda} d \langle F(\lambda)W_{+} x,
y\rangle. \tag{by  (\ref{equ3.1})}%
\end{align}
By the uniqueness of Fourier-Stieltjes transform, we have
\[
\langle E(\lambda) x, W_{+}^{*} y\rangle= \langle F(\lambda)W_{+} x, y\rangle,
\qquad\forall\ x, y\in\mathcal{H}, \ \forall\ \lambda\in\mathbb{R}.
\]
Thus
\begin{align}
W_{+}E(\lambda)=F(\lambda)W_{+}, \qquad\forall\ \lambda\in\mathbb{R}.
\label{equ6.4}%
\end{align}
Let $P\in\mathscr P_{ac}(H)$. Then
\begin{align}
(W_{+}P)^{*}F(\lambda)(W_{+}P)= PW_{+}^{*}F(\lambda)W_{+}P= PW_{+}^{*}%
W_{+}E(\lambda)P = PE(\lambda)P. \tag{by (\ref{equ6.4}) and (\ref{equ6.2})}%
\end{align}
This implies that the mapping $\lambda\mapsto(W_{+}P)^{*}F(\lambda)(W_{+}P)$
from $\mathbb{R}$ into $\mathcal{M}$ is locally absolutely continuous. By
Lemma \ref{lemma6.2.2}, we get that $R(W_{+}P)$, the range projection of
$W_{+}P$ in $\mathcal{M}$, is a subprojection of $P_{ac}^{\infty}(H_{1})$. So,
we obtain that $R(W_{+})\le P_{ac}^{\infty}(H_{1})$. Therefore,
\begin{align}
W_{+}W_{+}^{*}\le P_{ac}^{\infty}(H_{1}). \label{equ6.5}%
\end{align}

Similarly, as $H-H_{1}\in\mathcal{M}\cap L^{1}(\mathcal{M},\tau)$, we let
$V_{+}=sot\text{-}\lim\limits_{t\rightarrow\infty}e^{ itH}e^{-itH_{1}} P_{ac}^{\infty
}(H_{1}).$ Then
\begin{align}
\text{$V_{+}^{*}V_{+}=P_{ac}^{\infty}(H_{1})$ \qquad and \qquad$V_{+}V_{+}%
^{*}\le P_{ac}^{\infty}(H)$.} \label{equ6.6}%
\end{align}
We claim that
\begin{align}
\lim_{t\rightarrow\infty}\|(I-P_{ac}^{\infty}(H))e^{-itH_{1}} P_{ac}^{\infty
}(H_{1}) x\|=0, \qquad\forall\ x\in\mathcal{H}. \label{equ6.7}%
\end{align}
In fact, we have, for all $x\in\mathcal{H}$,
\begin{align}
\lim_{t\rightarrow\infty}\|(I-P_{ac}^{\infty}(H))  &  e^{-itH_{1}}
P_{ac}^{\infty}(H_{1}) x\|\nonumber\\
&  = \lim_{t\rightarrow\infty}\|(I-P_{ac}^{\infty}(H))e^{itH}e^{-itH_{1}}
P_{ac}^{\infty}(H_{1}) x\|\tag{by Proposition  \ref{prop6.2.3}}\\
&  =\|(I-P_{ac}^{\infty}(H))V_{+}x\|=0. \tag{by definition of $V_+$ and
(\ref{equ6.6})}%
\end{align}
Furthermore,
\begin{align}
W_{+} V_{+}  &  = \left(  sot\text{-}\lim_{t\rightarrow\infty}e^{ itH_{1}%
}e^{-itH} P_{ac}^{\infty}(H) \right)  \cdot\left(  sot\text{-}%
\lim_{t\rightarrow\infty}e^{ itH}e^{-itH_{1}} P_{ac}^{\infty}(H_{1})\right)
\nonumber\\
&  = sot\text{-} \lim_{t\rightarrow\infty}\left(  e^{ itH_{1}} P_{ac}^{\infty
}(H)e^{-itH_{1}} P_{ac}^{\infty}(H_{1}) \right) \tag{by Proposition
\ref{prop6.2.3}}\\
&  = P_{ac}^{\infty}(H_{1})+ sot\text{-} \lim_{t\rightarrow\infty}\left(  e^{
itH_{1}} (I- P_{ac}^{\infty}(H))e^{-itH_{1}} P_{ac}^{\infty}(H_{1}) \right)
\nonumber\\
&  = P_{ac}^{\infty}(H_{1}). \tag{by (\ref{equ6.7})}%
\end{align}
Thus $R(W_{+})= P_{ac}^{\infty}(H_{1})$, where $R(W_{+})$ is the range
projection of $W_{+}$ in $\mathcal{M}$. Combining with (\ref{equ6.5}), we
conclude that
\begin{align}
W_{+}W_{+}^{*} = P_{ac}^{\infty}(H_{1})\le P_{ac} (H_{1}). \label{equ6.8}%
\end{align}

Now from (\ref{equ6.2}), (\ref{equ6.3}), and
(\ref{equ6.8}), it follows that
\begin{align}
W_+  e^{itH}  W_+^*=    e^{itH_1} P_{ac}^{\infty}(H_{1}) , \qquad \forall t\in \mathbb R. \label{equ6.9} %
\end{align}
By Stone's Theorem (see Theorem 5.6.36 in \cite{Kadison}), we obtain from (\ref{equ6.9}) that
\begin{align}
W_+   H  W_+^*=    H_1  P_{ac}^{\infty}(H_{1}) .\label{equ6.10}%
\end{align}

The proof is now complete from (\ref{equ6.1}), (\ref{equ6.2}),
(\ref{equ6.8}), and
(\ref{equ6.10}).
\end{proof}

\begin{remark}
\label{rem5.2.5} Similarly it can be shown that, if $H$ and $H_{1}$ are
 self-adjoint elements in $\mathscr A(\mathcal{M})$ such that
$H_{1} - H \in\mathcal{M}\cap L^{1}(\mathcal{M},\tau), $ then
\[
W_{-}\triangleq sot\text{-}\lim_{t\rightarrow-\infty}e^{ itH_{1}}e^{-itH}
P_{ac}^{\infty}(H) \ \ \text{ exists in } \ \mathcal{M},
\]
and
\[
W_{-}^{*}W_{-}= P_{ac}^{\infty}(H)\le P_{ac} (H) \qquad\text{ and } \qquad
W_{-} W_{-}^{*}= P_{ac}^{\infty}(H_{1}) \le P_{ac} (H_{1}).
\]

\end{remark}

\begin{remark}
By Lemma  \ref{lemma5.2.2.1-28}, Theorem \ref{thm5.2.2} and Remark
\ref{rem5.2.5} imply the classical Kato-Rosenblum Theorem for self-adjoint
operators in $\mathcal{B}(\mathcal{H})$ when $\mathcal H$ is separable.
\end{remark}

\begin{example}
Let $\mathcal{N}$ be a finite von Neumann algebra with a faithful normal
tracial state. Let $\mathcal{H}_{0}=L^{2}(\mathbb{R}^{3}, \mu)$, where $\mu$
is the Lebesgue measure on $\mathbb{R}^{3}$. Let $\mathcal{M}=\mathcal{N}%
\otimes\mathcal{B}(\mathcal{H}_{0})$, acting naturally on the Hilbert space
$\mathcal{H}=L^{2}(\mathcal{N},\tau)\otimes\mathcal{H}_{0}$, be a semifinite
von Neumann algebra. Let $\Delta$ be the Laplacian operator on $L^{2}%
(\mathbb{R}^{3}, \mu)$. Then $-(I_{\mathcal{N}}\otimes\Delta)$ is a densely
defined, self-adjoint operator in $\mathscr A(\mathcal{M})$. As $-\Delta$ is
spectrally absolutely continuous on $L^{2}(\mathbb{R}^{3}, \mu)$ (see Section X.3.4 in
\cite{Kato4}), Proposition \ref{example5.1.3} shows that $P_{ac}^{\infty
}(-(I_{\mathcal{N}}\otimes\Delta))=I_{\mathcal{M}}$.
\end{example}

\subsection{Real part of hyponormal operators}

When $H$ is bounded,   the following analogue of Kato-Putnam criterion in
\cite{Kato3} and \cite{Putnam} is sometimes useful  to determine whether
$P_{ac}^\infty(H)$ is zero or not.

\begin{proposition}\label{prop5.2.9}
Assume that $H$ is a self-adjoint element in $\mathcal M$. Then the following statements are equivalent.
\begin{enumerate}
 \item [(i)]
  $P_{ac}^\infty(H)\ne 0$.
 \item [(ii)] There exist a self-adjoint element $K\in\mathcal M$ and a
 nonzero positive element $L\in\mathcal M$ such that $$i(HK-KH)=L.$$
\end{enumerate}
\end{proposition}
\begin{proof}
(i)$\Rightarrow$(ii) As $P_{ac}^\infty(H)\ne 0$, there exists a nonzero projection $P\in \mathscr P_{ac}^\infty(H)$. Let $\omega_n(\cdot)$ be as in Definition \ref{def4.1.5}. By Lemma \ref{lemma4.2.3} (vi),
\[
\int_{\mathbb{R}} \|Pe^{-i\lambda H}\omega_{{n}}(H) x\|^{2} d\lambda \le\frac{n } {2\pi} \|x\|^{2}, \qquad \text{ for all } x\in\mathcal H.
\]Since $e^{-i\lambda H}$ and $\omega_{{n}}(H)$ commute, by Theorem 2.1 in \cite{Kato3},
$$
\sup_{\lambda>\mu}\frac { \|P\omega_{{n}}(H)
(E(\lambda)-E(\mu))\|^2}{\lambda-\mu}<\infty,
$$ where $\{E(\lambda)\}_{\lambda\in\mathbb R}$ is the spectral resolution of the identity for $H$ in $\mathcal M$.
Hence $$ \sup_{\lambda>\mu}\frac { \|\omega_{{n}}(H)P\omega_{{n}}(H)
(E(\lambda)-E(\mu))\|^2}{\lambda-\mu}\le \sup_{\lambda>\mu}\frac {
\|P\omega_{{n}}(H) (E(\lambda)-E(\mu))\|^2}{\lambda-\mu}<\infty.
$$ By Lemma \ref{lemma5.2.2.1}, we might assume that $\omega_{{n}}(H)P\omega_{{n}}(H)\ne 0$ for a large $n\in \mathbb N$. Again from  Theorem 2.1 in \cite{Kato3}, $\omega_{{n}}(H)P\omega_{{n}}(H)$ is a nonzero positive $H$-smooth operator in $\mathcal M$. Let
$$
L= (\omega_{{n}}(H)P\omega_{{n}}(H))^2.
$$
Theorem 3.2 of \cite{Kato3} shows that
$$
K=   - \int^{\infty}_0 e^{i\lambda H}Le^{-i\lambda H} d\lambda \qquad \text{(convergence is in strong operator topology)}.
$$
  exists in $\mathcal M$  and  Lemma X.5.2 in \cite {Kato4} implies that $i(HK-KH)=L.$

(ii)$\Rightarrow$ (i) Assume that $K, L$ are elements  in $\mathcal
M$ with desired properties. From Theorem 6.2 in \cite{Kato3}, it
follows that $L^{1/2}$ is $H$-smooth. Now Theorem 2.1 in
\cite{Kato3} implies that the mapping $\lambda\rightarrow
L^{1/2}E(\lambda)L^{1/2}$ from $\mathbb R$ into $\mathcal M$ is
Lipschitz continuous (so locally absolutely continuous). By Lemma
\ref{lemma6.2.2}, the range projection of $L^{1/2}$ is a nonzero
subprojection of $P_{ac}^\infty(H)$. Therefore $P_{ac}^\infty(H)\ne
0$, which ends the proof of the proposition.
\end{proof}

Recall that an operator $T$ in $\mathcal B(\mathcal H)$ is hyponormal if
$T^*T-TT^*$ is positive. The following result could be compared to Corollary VI.3.3 in \cite{Martin}.
\begin{corollary}  \label{prop5.2.10}
 $H$ is a self-adjoint element in $\mathcal
M$ with  $P_{ac}^\infty(H)\ne 0$ if and only if $H$ is the real part
of a non-normal hyponormal operator $T$ in $\mathcal M$.
\end{corollary}
\begin{proof}
It is a direct consequence of Proposition \ref{prop5.2.9}.
\end{proof}

\begin{remark}
Proposition \ref{prop5.2.9} may be used to construct  new examples
of self-adjoint operators with nonzero norm absolutely continuous
projections in $\mathcal M$. For example, let $H, K$ and $L$ be as
above. Assume $K_1$ is a self-adjoint element in $\mathcal M$ such
that $K_1K=KK_1$. Then $i((H+K_1)K-K(H+K_1))=L$. Hence, $H+K_1$ is a
self-adjoint element in $\mathcal M$ with nonzero norm absolutely
continuous projections.
\end{remark}

\begin{proposition}\label{prop5.3.4}
\label{cor5.2.5}  If $H$ is a self-adjoint element in $\mathcal{M}$ such that $P_{ac}%
^{\infty}(H)\ne0$, then there exists no self-adjoint diagonal
operator $K$ in $\mathcal{M}$ satisfying $H-K\in
L^{1}(\mathcal{M},\tau)$. In particular, if $H$ is  the real part of a
non-normal hyponormal operator   in $\mathcal M$, then there exists no self-adjoint diagonal
operator $K$ in $\mathcal{M}$ satisfying $H-K\in
L^{1}(\mathcal{M},\tau)$.
\end{proposition}

\begin{proof}
Note that, if $K$ is a self-adjoint diagonal operator in
$\mathcal{M}$, then $P^{\infty}_{ac}(K)\le P_{ac}(K)=0$. Now that
result is a direct consequence of Theorem \ref{thm5.2.2} and
Corollary \ref{prop5.2.10}.
\end{proof}

\subsection{Kuroda-Birman Theorem in semifinite von Neumann algebras}

\begin{lemma}
\label{lemma5.3.1} Suppose $H$ is a  self-adjoint operator in $\mathscr
A(\mathcal{M})$. 

\begin{enumerate}
\item[(i)] If $A\in\mathcal{M}\cap L^{2}(\mathcal{M},\tau)$, then
$sot\text{-}\lim\limits_{t\rightarrow\infty} \left(  Ae^{-itH}P_{ac}^{\infty}(H)
\right)  =0.$

\item[(ii)] If $B\in\mathcal{M}\cap L^{1}(\mathcal{M},\tau)$, then
$sot\text{-}\lim\limits_{t\rightarrow\infty}\left(  Be^{-itH}P_{ac}^{\infty}(H)
\right)  =0.$
\end{enumerate}
\end{lemma}

\begin{proof}
(i) Let $A\in\mathcal{M}\cap L^{2}(\mathcal{M},\tau)$. If suffices to show
that, if $P\in\mathscr P_{ac}^{\infty}(H)$, then
\[
sot\text{-}\lim_{t\rightarrow\infty} \left(  Ae^{-itH}P \right)  =0.
\]
By Lemma \ref{lemma3.1.1}, there exists a family $\{x_{m}\}_{m\in\mathbb{N}}$
of vectors in $\mathcal{H}$ such that
\begin{equation}
\|X\|_{2}^{2}=\tau(X^{*}X)= \sum_{m} \langle X^{*}X x_{m}, x_{m}\rangle=
\sum_{m} \|Xx_{m}\|^{2},\qquad\forall\ X\in\mathcal{M}. \label{equFeb4_5.8}
\end{equation}
 Assume $P\in\mathscr P_{ac}^{\infty}(H)$. Let $\omega_n(\cdot)$ be as in Definition \ref{def4.1.5}. By Lemma \ref{lemma4.2.3} (v) and the Riemann-Lebesgus Lemma (also see Theorem 1.8.1 of \cite {Arendt}),
\begin{equation*}
\lim_{t\rightarrow \pm\infty} \|Pe^{-itH}\omega_n(H)x\|=0, \qquad \text{ for all } x\in\mathcal H.
\end{equation*}
From Lemma \ref{lemma5.2.2.1}, it follows that
\begin{equation}
\lim_{t\rightarrow \pm\infty} \|Pe^{-itH} x\|=0, \qquad \text{ for all } x\in\mathcal H.\label{equFeb4_5.9}
\end{equation}

We claim that \begin{equation}
\lim_{t\rightarrow \pm\infty} \|Ae^{-itH}P\|_2= \lim_{t\rightarrow \mp\infty} \|Pe^{-itH}A^*\|_2=0. \label{equFeb4_5.10}
\end{equation}
In fact, by  (\ref{equFeb4_5.8}), we have \begin{equation}
\|Pe^{-itH}A^*\|_2^2= \sum_m \|Pe^{-itH}A^*x_m\|^2. \label{equFeb4_5.11}
\end{equation} It induces from (\ref{equFeb4_5.9}) that \begin{equation}
\lim_{t\rightarrow \pm\infty} \|Pe^{-itH} A^*x_m\|=0, \qquad \text { for all } m\in\mathbb N.  \label{equFeb4_5.12}
\end{equation}
Note that  \begin{equation}
\|Pe^{-itH} A^*x_m\|\le \|A^*x_m\| \quad \text { and } \quad \sum_m \|Pe^{-itH}A^*x_m\|^2 \le \sum_m \|A^*x_m\|^2 =\|A^*\|_2^2<\infty.\label{equFeb4_5.13}
\end{equation}
By Lebesgue Dominating Convergence Theorem, from (\ref{equFeb4_5.11}), (\ref{equFeb4_5.12}) and (\ref{equFeb4_5.13}), we conclude that
$$
\lim_{t\rightarrow \pm\infty} \|Ae^{-itH}P\|_2= \lim_{t\rightarrow \mp\infty} \|Pe^{-itH}A^*\|_2=0.
$$
I.e. (\ref{equFeb4_5.10}) holds.

From  (\ref{equFeb4_5.8}) and  (\ref{equFeb4_5.10}), it follows that
\[
\lim_{t\rightarrow\infty} \|Ae^{-itH}Px_{m}\| \le \lim_{t\rightarrow\infty} \|Ae^{-itH}P\|_2=0, \qquad\forall\ m\in
\mathbb{N}.
\]
Furthermore,
\[
\lim_{t\rightarrow\infty} \|Ae^{-itH}PA^{\prime}x_{m}\|=\lim_{t\rightarrow\infty} \|A^{\prime} (Ae^{-itH}Px_{m})\|=0, \qquad
\forall\ A^{\prime}\in\mathcal{M}^{\prime}, \ \text{ and } \ m\in\mathbb{N}.
\]
By Lemma \ref{lemma3.1.1},
\[
\lim_{t\rightarrow\infty} \|Ae^{-itH}Px\|=0, \qquad\forall\ x\in\mathcal{H},
\]
i.e.
\[
sot\text{-}\lim_{t\rightarrow\infty} \left(  Ae^{-itH}P \right)  =0.
\]
This ends the proof of (i).

(ii) follows from the fact that $M\cap L^{1}(\mathcal{M},\tau)\subseteq M\cap
L^{2}(\mathcal{M},\tau).$
\end{proof}

The following is an analogue of Kuroda-Birman Theorem in a semifinite von
Neumann algebra.

\begin{theorem}
\label{thm5.3.2} Suppose $H$ and $H_{1}$ are  self-adjoint elements in
$\mathscr A(\mathcal{M})$ such that
\[
(H_{1}+i)^{-1} - (H+i)^{-1} \in\mathcal{M}\cap L^{1}(\mathcal{M},\tau).
\]
Then
\[
W_{+}\triangleq sot\text{-}\lim_{t\rightarrow\infty}e^{ itH_{1}}e^{-itH}
P_{ac}^{\infty}(H) \ \ \text{ exists in } \ \mathcal{M}.
\]
Moreover, \begin{enumerate}
\item [(i)]$\displaystyle
W_{+}^{*}W_{+}= P_{ac}^{\infty}(H)\le P_{ac} (H)$ and $ \displaystyle W_{+} W_{+}^{*}= P_{ac}^{\infty}(H_{1}) \le P_{ac} (H_{1}).
$
\item [(ii)]
$\displaystyle
W_+   HP_{ac}^{\infty}(H) W_+^*=  H_1 P_{ac}^{\infty}(H_{1})  .
$
\end{enumerate}

\end{theorem}

\begin{proof}
Firstly, we will show that
\[
W_{+}\triangleq sot\text{-}\lim_{t\rightarrow\infty}e^{ itH_{1}}e^{-itH}
P_{ac}^{\infty}(H) \ \ \text{ exists in } \ \mathcal{M}.
\]
In fact, we let $J=(H_{1}+i)^{-1}(H+i)^{-1} $ and $B= -(H_{1}+i)^{-1} +
(H+i)^{-1}$. Then $J\mathcal{D}(H)\subseteq\mathcal{D}(H_{1})$ and
\[
H_{1}J-JH= (H_{1}+i-i)J-J(H+i-i)= B\in\mathcal{M}\cap L^{1}(\mathcal{M}%
,\tau).
\]
By Theorem \ref{thm5.0.3} and the definition of $P_{ac}^{\infty}(H)$,
\begin{equation}
sot\text{-}\lim_{t\rightarrow\infty}e^{ itH_{1}}Je^{-itH} P_{ac}^{\infty}(H)
\ \ \text{ exists in } \ \mathcal{M}. \label{equa6.9}%
\end{equation}
Proposition \ref{prop6.2.3} implies that $(H+i)^{-1}$ commutes with
$P_{ac}^{\infty}(H)$. Combining with (\ref{equa6.9}), we conclude
\[
\lim_{t\rightarrow\infty}e^{ itH_{1}}(H_{1}+i)^{-1}e^{-itH} P_{ac}^{\infty}(H)x
\ \ \text{ exists in $\mathcal{H}$ for all $x\in\mathcal{D}(H)$} .
\]
Note that $\mathcal{D}(H)$ is dense in $\mathcal{H}$. We have
\begin{equation}
\lim_{t\rightarrow\infty}e^{ itH_{1}}(H_{1}+i)^{-1}e^{-itH} P_{ac}^{\infty}(H)x
\ \ \text{ exists in $\mathcal{H}$ for all $x\in\mathcal{H}$} .
\label{equa5.10}%
\end{equation}
Combining (\ref{equa5.10}), Lemma \ref{lemma5.3.1} with the fact that
$(H_{1}+i)^{-1}=(H+i)^{-1}-B$, we get
\begin{align}
\lim_{t\rightarrow\infty} e^{ itH_{1}}  &  (H+i)^{-1}e^{-itH} P_{ac}^{\infty
}(H)x\nonumber\\
&  =\lim_{t\rightarrow\infty}e^{ itH_{1}}\left(  (H_{1}+i)^{-1}+B\right)
e^{-itH} P_{ac}^{\infty}(H)x\nonumber\\
&  = \lim_{t\rightarrow\infty} e^{ itH_{1}}(H_{1}+i)^{-1}e^{-itH}
P_{ac}^{\infty}(H)x \ \ \text{ exists in $\mathcal{H}$ for all $x\in
\mathcal{H}$} . \label{equa5.11}%
\end{align}
As $(H+i)^{-1}$ commutes with $P_{ac}^{\infty}(H)$ and $\mathcal{D}(H)$ is
dense in $\mathcal{H}$, it follows from (\ref{equa5.11}) that
\[
\lim_{t\rightarrow\infty}e^{ itH_{1}} e^{-itH} P_{ac}^{\infty}(H)x \ \ \text{
exists in $\mathcal{H}$ for all $x\in\mathcal{H}$} ,
\]
i.e.
\[
W_{+}\triangleq sot\text{-}\lim_{t\rightarrow\infty}e^{ itH_{1}}e^{-itH}
P_{ac}^{\infty}(H) \ \ \text{ exists in } \ \mathcal{M}.
\]

Secondly, since the proof of $W_{+}^{*}W_{+}= P_{ac}^{\infty}(H) $, $W_{+}
W_{+}^{*}= P_{ac}^{\infty}(H_{1}) $ and $\displaystyle
W_+   H  W_+^*=     H_1 P_{ac}^{\infty}(H_{1})
$ is similar to the one in Theorem
\ref{thm5.2.2}, it is skipped.
\end{proof}

\begin{remark}
Similarly it can be shown that, { if $H$ and $H_{1}$ are densely defined,
self-adjoint elements in $\mathscr A(\mathcal{M})$ such that $(H_{1}+i)^{-1} -
(H+i)^{-1} \in\mathcal{M}\cap L^{1}(\mathcal{M},\tau), $ then
\[
W_{-}\triangleq sot\text{-}\lim_{t\rightarrow-\infty}e^{ itH_{1}}e^{-itH}
P_{ac}^{\infty}(H) \ \ \text{ exists in } \ \mathcal{M},
\]
and
\[
W_{-}^{*}W_{-}= P_{ac}^{\infty}(H)\le P_{ac} (H) \qquad\text{ and } \qquad
W_{-} W_{-}^{*}= P_{ac}^{\infty}(H_{1}) \le P_{ac} (H_{1}).
\]
}
\end{remark}

\section{Small Perturbation of Bounded Self-adjoint Operators}

Let $\mathcal{M}$ be a countably decomposable, properly infinite, semifinite
von Neumann algebra acting on a Hilbert space $\mathcal{H}$. Let $\tau$ be a
faithful normal semifinite tracial weight of $\mathcal{M}$.

\subsection{Voiculescu's constant}

Let $\mathcal{K}_{\Phi}(\mathcal{M},\tau)$ be a norm ideal of $\mathcal{M}$
(see Definition 2.1.1 in \cite{Li} for its definition). Recall $\mathcal{F}%
(\mathcal{M},\tau) $ is the set of finite rank operators in $(\mathcal{M}%
,\tau)$. Then $\mathcal{K}_{\Phi}^{0}(\mathcal{M},\tau)$ is the completion of
$\mathcal{F}(\mathcal{M},\tau) $ with respect to the norm $\Phi$ in
$\mathcal{K}_{\Phi}(\mathcal{M},\tau)$. We further let $\mathcal{K}_{\Phi}%
^{0}(\mathcal{M},\tau)_{1}^{+}$ be the unit ball of positive elements in
$\mathcal{K}_{\Phi}^{0}(\mathcal{M},\tau)$.

Following Voiculescu's Definition in \cite{Voi3.1}, we introduce the following concept.

\begin{definition}
\label{def6.1.1} Let $n\in\mathbb{N}$ and $X_{1},\ldots, X_{n}$ be an
$n$-tuple of elements in $\mathcal{M}$. We define
\[
{\mathscr K}_{\Phi}(X_{1},\ldots, X_{n}; \mathcal{M},\tau)=\liminf
_{A\in\mathcal{K}_{\Phi}^{0}(\mathcal{M},\tau)_{1}^{+}}\left(  \max_{1\le i\le
n}\Phi(AX_{i}-X_{i}A) \right)  ,
\]
where the $\liminf$'s are taken with respect to the natural orders on
$\mathcal{K}_{\Phi}^{0}(\mathcal{M},\tau)_{1}^{+}$.

When $\mathcal{K}_{\Phi}(\mathcal{M},\tau)$ is $\mathcal{M}\cap L^{p}%
(\mathcal{M},\tau)$ with
\[
\text{$\Phi(X)=\max\{\|X\|, \|X\|_{p}\}$, \qquad$\forall X\in\mathcal{M}$,}%
\]
for some $1\le p<\infty$, we will denote ${\mathscr K}_{\Phi}(X_{1},\ldots,
X_{n}; \mathcal{M},\tau)$, $\mathcal{K}_{\Phi}(\mathcal{M},\tau)$ by
${\mathscr K}_{p}(X_{1},\ldots, X_{n}; \mathcal{M},\tau)$, $\mathcal{K}%
_{p}(\mathcal{M},\tau)$ respectively.
\end{definition}

The next lemma is a direct consequence of preceding definition.

\begin{lemma}
\label{lemma6.1.2} Let $\mathcal{M}$ be a countably decomposable, properly
infinite von Neumann algebra with a faithful normal semifinite tracial weight
$\tau$. Let $\mathcal{K}_{\Phi}(\mathcal{M},\tau)$ be a norm ideal of
$\mathcal{M}$. Let $n\in\mathbb{N}$ and $X_{1},\ldots, X_{n}$ be an $n$-tuple
of elements in $\mathcal{M}$.

If, for any $\epsilon>0$, there exists a family of commuting diagonal
operators $D_{1},\ldots,D_{n}$ in $\mathcal{M}$ such that $\max_{1\le i\le
n}\Phi(X_{i}-D_{i})<\epsilon$, then ${\mathscr K}_{\Phi}(X_{1},\ldots,
X_{n})=0.$
\end{lemma}

\subsection{An Example}

If $\mathcal{M}_{1}$ is a von Neumann subalgebra of $\mathcal{M}$ such that
the restriction of $\tau$ to $\mathcal{M}_{1}$ is semifinite, then there
exists a faithful normal trace-preserving conditional expectation
$\mathcal{E}$ from $\mathcal{M}$ onto $\mathcal{M}_{1}$ (see Definition IX.4.1
of \cite{Takesaki}). For each $1\leq p<\infty$, the conditional expectation
$\mathcal{E}$ induces a contraction, still denoted by $\mathcal{E}$, from
$L^{p}(\mathcal{M},\tau)$ onto $L^{p}(\mathcal{M}_{1},\tau)$ satisfying
$\mathcal{E}(AXB)=A\mathcal{E}(X)B$ for all $A,B\in\mathcal{M}_{1}$ and $X\in
L^{p}(\mathcal{M},\tau)$ (see \cite{Pisier}).

\begin{proposition}
\label{prop6.2.1} Let $\mathcal{M}$ be a countably decomposable, properly
infinite von Neumann algebra with a faithful normal semifinite tracial weight
$\tau$. Suppose $\mathcal{M}_{1}$ is a von Neumann subalgebra of $\mathcal{M}$
such that the restriction of $\tau$ to $\mathcal{M}_{1}$ is semifinite. If
$X_{1},\ldots,X_{n}$ is an $n$-tuple of elements in $\mathcal{M}_{1}$, then
\[
{\mathscr K}_{p}(X_{1},\ldots,X_{n};\mathcal{M}_{1},\tau)={\mathscr K}%
_{p}(X_{1},\ldots,X_{n};\mathcal{M},\tau),\qquad\forall\ 1\leq p<\infty.
\]

\end{proposition}

\begin{proof}
Let $1\le p<\infty$. Let $\mathcal{K}_{p} (\mathcal{M},\tau)= \mathcal{M}\cap
L^{p}(\mathcal{M},\tau)$ be equipped with a $\max\{\|\cdot\|, \|\cdot\|_{p}%
\}$-norm. We should note that $\mathcal{K}_{p}^{0}(\mathcal{M},\tau)
=\mathcal{K}_{p}(\mathcal{M},\tau) $ in this case.

It is obvious that
\[
{\mathscr K}_{p}(X_{1},\ldots, X_{n}; \mathcal{M}_{1}, \tau) \ge
{\mathscr K}_{p}(X_{1},\ldots, X_{n}; \mathcal{M}, \tau).
\]
We need only to show that
\[
{\mathscr K}_{p}(X_{1},\ldots, X_{n}; \mathcal{M}_{1}, \tau) \le
{\mathscr K}_{p}(X_{1},\ldots, X_{n}; \mathcal{M}, \tau).
\]

Let $\epsilon>0$ be given. By definition, there exists an increasing sequence
$\{A_{m}\}_{m\in\mathbb{N}}$ in $\mathcal{K}_{p}(\mathcal{M},\tau)_{1}^{+}$
such that $A_{m}$ converges to $I$ in strong operator topology and
\[
\max\{ \|A_{m}X_{i}-X_{i}A_{m}\| , \|A_{m}X_{i}-X_{i}A_{m}\|_{p}\}<
{\mathscr K}_{p}(X_{1},\ldots, X_{n}; \mathcal{M}, \tau) +\epsilon,
\]
for all $1\le i\le n$ and $m\in\mathbb{N}$.

As the restriction of $\tau$ to $\mathcal{M}_{1}$ is semifinite, there exist a
faithful normal trace-preserving conditional expectation $\mathcal{E}$ from
$\mathcal{M}$ onto $\mathcal{M}_{1}$ and an induced contraction $\mathcal{E}$
from $L^{p}(\mathcal{M},\tau)$ onto $L^{p}(\mathcal{M}_{1},\tau)$. Therefore,
$\{\mathcal{E}(A_{m})\}_{m\in\mathbb{N}}$ is an increasing sequence in
$\mathcal{M}_{1}$ such that $\{\mathcal{E}(A_{m})\}_{m\in\mathbb{N}}$
converges to $I$ in weak operator topology  (so, in strong operator topology) and
\[
\begin{aligned}
{\mathscr K}_p(X_1,\ldots, X_n; \mathcal M, \tau) +\epsilon & > \max \{ \|\mathcal E(A_mX_i-X_iA_m)\| , \|\mathcal E(A_mX_i-X_iA_m)\|_p\} \\
&=  \max \{ \|\mathcal E(A_m)X_i-X_i\mathcal E(A_m)\| , \|\mathcal E(A_m) X_i-X_i\mathcal E(A_m)\|_p\}.\end{aligned}
\]
Note that $0\le\mathcal{E}(A_{m})\le I$ and $\|\mathcal{E}(A_{m})\|_{p}
\le\|A_{m}\|_{p}<\infty$. Thus, $\mathcal{E}(A_{m})\in\mathcal{K}%
_{p}(\mathcal{M},\tau)_{1}^{+}$. By definition, we have
\[
\begin{aligned}
{\mathscr K}_p(X_1,\ldots, X_n; \mathcal M_1, \tau ) &\le \liminf_m  \max_{1\le i\le n} \{ \|A_mX_i-X_iA_m\| , \|A_mX_i-X_iA_m\|_p\} \\&\le  {\mathscr K}_p(X_1,\ldots, X_n; \mathcal M, \tau) +\epsilon.\end{aligned}
\]
As $\epsilon$ is arbitrary, we have
\[
{\mathscr K}_{p}(X_{1},\ldots, X_{n}; \mathcal{M}_{1}, \tau) \le
{\mathscr K}_{p}(X_{1},\ldots, X_{n}; \mathcal{M}, \tau).
\]
This completes the proof of the proposition.
\end{proof}

\begin{example}\label{example6.2.2}
Let $\mathcal{N}$ be a finite von Neumann algebra with a faithful normal
tracial state $\tau_{\mathcal{N}}$ and let $\mathcal{H}_{0}$ be an infinite
dimensional separable Hilbert space. Then $\mathcal{M}=\mathcal{N}%
\otimes\mathcal{B}(\mathcal{H}_{0})$ is a semifinite von Neumann
algebra with a faithful normal tracial weight
$\tau_{\mathcal{M}}=\tau_{\mathcal{N}} \otimes Tr$, where $Tr$ is
the canonical trace of $\mathcal{B}(\mathcal{H}_{0})$. We might
further assume that $\mathcal{M}$ acts naturally on the Hilbert
space
$\mathcal{H}=L^{2}(\mathcal{N},\tau_{\mathcal{N}})\otimes\mathcal{H}_{0}$.
Obviously, $I_{\mathcal{N}}\otimes\mathcal{B}(\mathcal{H}_0)$ is a
von Neumann subalgebra of $\mathcal{M}$ such that the restriction of
$\tau_{\mathcal{M}}$ on
$I_{\mathcal{N}}\otimes\mathcal{B}(\mathcal{H}_0)$ is semifinite.

Let $n\geq2$ be a positive integer. By Proposition 4.1 in \cite{Voi}, there
exists an $n$-tuple $X_{1},\ldots,X_{n}$ of commuting self-adjoint elements in
$\mathcal{B}(\mathcal{H}_{0})$ such that
\[
k_{p}(X_{1},\ldots,X_{n})>0,\qquad\forall\ 1\leq p<n,
\]
where $k_{p}(X_{1},\ldots,X_{n})$ is a constant defined in section 1 of
\cite{Voi}. By Proposition 1.1 in \cite{Voi6} and Definition \ref{def6.1.1},
\[
{\mathscr K}_{p}(X_{1},\ldots,X_{n};\mathcal{B}(\mathcal{H}_{0}),Tr)={k}_{p}%
(X_{1},\ldots,X_{n}),\qquad\forall\ 1\leq p<n.
\]
By Proposition \ref{prop6.2.1},
\[
\begin{aligned}{\mathscr K}_p(I_{\mathcal N}\otimes X_1,\ldots, I_{\mathcal N}\otimes X_n; \mathcal M,\tau) &=
{\mathscr K}_p(I_{\mathcal N}\otimes X_1,\ldots, I_{\mathcal N}\otimes X_n; I_{\mathcal N}\otimes \mathcal B(\mathcal H_{0}),1\otimes Tr) \\
&= {\mathscr K}_p(X_1,\ldots, X_n; \mathcal B(\mathcal H_{0}),Tr) \\
& =k_p(X_1,\ldots, X_n)\\
&> 0, \qquad \forall \ 1 \le p<n.
\end{aligned}
\]
By Lemma \ref{lemma6.1.2}, $I_{\mathcal{N}}\otimes X_{1},\ldots,I_{\mathcal{N}%
}\otimes X_{n}$ is a family of commuting self-adjoint elements in
$\mathcal{M}$ that are not small perturbations of commuting diagonal operators
modulo $\mathcal{M}\cap L^{p}(\mathcal{M},\tau)$ for all $1\leq p<n$.
\end{example}

\vspace{1cm}

\end{document}